\documentclass[12pt, a4paper]{amsart}

\usepackage{amsmath, amssymb, amsfonts, amsthm}
\usepackage[margin=2cm]{geometry}
\usepackage[]{setspace}
\usepackage[usenames,dvipsnames]{color}
\usepackage[,unicode]{hyperref}
\usepackage{verbatim}
\hypersetup{
    urlcolor=blue,
    colorlinks=true,
    linkcolor= blue,
    citecolor=blue,
}
\usepackage{mathtools}
\usepackage[all,matrix,arrow]{xy}
\usepackage{verbatim}
\usepackage{microtype}

\theoremstyle{plain}
\newtheorem{theorem}{Theorem}[section]
\newtheorem*{mtheorem}{Main Theorem}
\newtheorem{proposition}[theorem]{Proposition}
\newtheorem{lemma}[theorem]{Lemma}
\newtheorem{corollary}[theorem]{Corollary}

\theoremstyle{definition}
\newtheorem{example}[theorem]{Example}

\newtheorem{definition}[theorem]{Definition}

\theoremstyle{remark}
\newtheorem{remark}[theorem]{Remark}
\newtheorem{specialremark}{Remark}

\usepackage{tikz}
\usetikzlibrary{shapes.geometric, arrows}
\tikzstyle{startstop} = [rectangle, rounded corners, minimum width=3cm, minimum height=1cm,text centered, draw=black, fill=red!30]
\tikzstyle{io} = [trapezium, rounded corners, trapezium left angle=70, trapezium right angle=110, minimum width=3cm, minimum height=1cm, text centered, draw=black, fill=blue!30]
\tikzstyle{process} = [rectangle, rounded corners, minimum width=3cm, minimum height=1cm, text centered, draw=black, fill=orange!30]
\tikzstyle{decision} = [rectangle, rounded corners, minimum width=3cm, minimum height=1cm, text centered, draw=black, fill=green!30]
\tikzstyle{arrow} = [thick,->,>=stealth]

\newcommand\res{\mathrm{res}}

\DeclareMathOperator{\Th}{Th}

\DeclareMathOperator{\tp}{tp}
\newcommand{\Lring}{\mathcal{L}_{\mathrm{ring}}}
\newcommand{\Lval}{{\mathcal{L}_{\mathrm{val}}}}
\newcommand{\Loag}{\mathcal{L}_{\mathrm{oag}}}
\newcommand{\Ldagger}[1]{{\mathcal{L}_{p, e}}}

\providecommand{\rmOmega}{\mathrm{\Omega}}

\newcommand{\I}[1][]{{\bf(i#1)}}

\title[AKE principles for finitely ramified henselian fields]{Ax--Kochen--Ershov principles for finitely ramified henselian fields}
\author{Sylvy Anscombe}
\address{Universit\'{e} Paris Cit\'{e} and Sorbonne Universit\'{e}, CNRS, IMJ-PRG, F-75013 Paris, France}
\email{sylvy.anscombe@imj-prg.fr}
\author{Philip Dittmann}
\address{Institut f\"{u}r Algebra, Technische Universit\"{a}t Dresden, 01062 Dresden, Germany}
\curraddr{Department of Mathematics, University of Manchester, Manchester M13 9PL, United Kingdom}
\email{philip.dittmann@manchester.ac.uk}
\author{Franziska Jahnke}
\address{Institut f\"{u}r Mathematische Logik und Grundlagenforschung,
University of M\"{u}nster,
Einsteinstr. 62,
48149 M\"{u}nster,
Germany}
\address{Institute for Logic, Language and Computation (ILLC),          
University of Amsterdam,
Science Park 107,
1098 XG Amsterdam, Netherlands}
\email{franziska.jahnke@uni-muenster.de}

\thanks{This is a version of this paper with a small update
  compared to the version published in the Transactions of the American Mathematical Society.
  See Remark~\ref{specrem:muensterlemma}.
  The numbering is the same as in the published version
}

\begin{document}
\begin{abstract}
We study the model theory of finitely ramified henselian valued fields of 
fixed initial ramification, obtaining versions
of the Ax--Kochen--Ershov principle as follows.
We identify the induced structure on the residue field and show that
once the residue field is endowed with this structure, the theory of the valued field
is determined by the theories of the enriched residue field and the value group. Similarly, we show
that the existential theory of the valued field is determined by the positive existential
theory of the enriched residue field. 
We also prove that an embedding of finitely ramified henselian valued fields 
is existentially closed as soon as the induced embeddings of value group and residue
field are existentially closed. This last result requires no enrichment of the residue field, 
in analogy to the corresponding result for model completeness, which holds by results of Ershov and Ziegler. 
\end{abstract}

\maketitle

\section{Introduction}
The aim of this paper is to give a number of Ax--Kochen--Ershov Theorems
for finitely ramified henselian valued fields: reducing questions of 
elementary equivalence, existential equivalence, elementarity of substructures, 
and existential closedness to corresponding properties of residue fields
and value groups.
In other words, we prove completeness, existential completeness and model completeness results for theories of finitely ramified henselian valued fields relative to residue fields and value groups.

Throughout the paper, we consider valued fields as structures 
in the three-sorted
language $\mathcal{L}_\mathrm{val}$ (see Section \ref{sec:preliminaries} for precise definitions of the
languages we use).
A valued field $(K,v)$ of mixed characteristic $(0,p)$ is called
\emph{finitely ramified} if the value group interval $(0,v(p)]$ 
is finite. The cardinality of this interval is called the \emph{initial ramification}.
Distinguished among finitely ramified valued fields
are the unramified ones, i.e., those where $v(p)$ is minimum positive in the value group.
For unramified henselian valued fields,
several Ax--Kochen--Ershov Theorems were recently proven by Anscombe and
Jahnke \cite{AJCohen},
with a number of predecessors for perfect residue fields\footnote{Without
  restricting to perfect 
residue fields, relative completeness was already
claimed by Ershov in \cite[Theorem 4.3.6]{ErshovMult} and by B\'elair in \cite[Corollaire 5.2]{Belair}. Moreover,
Ershov states an AKE principle for existential closedness \cite[Theorem 4.3.5]{ErshovMult}. However, all these proofs
rely on Witt vector techniques which are only available in the case of perfect residue fields, see~\cite[Proof of Theorem 4.1.3']{ErshovMult} and \cite[Proof of Corollaire 5.2]{Belair}.}
going back to \cite{AxKochen-II} and \cite{Er65}.

The model theory of finitely ramified henselian valued fields was first studied independently
by Ershov (see \cite{ErshovMult} for an English version) and Ziegler \cite{ZieglerDiss}. 
They show that an embedding of
finitely ramified henselian valued fields is elementary if and only if the induced
embeddings of value group and residue field are elementary.
However, without further structure on the residue field, reducing 
elementary equivalence and existential equivalence to corresponding properties of residue field
and value group fails: this is illustrated in Examples \ref{ex:notAKE} and \ref{ex:hiding} below.
It was previously shown in \cite{Ditt22} (see Example \ref{ex:hidex}) that even a transfer of existential decidability fails.

Therefore, in the present article we regard the residue fields of the valued fields under consideration as $\Ldagger{e}$-structures, where $\Ldagger{e}$ is an expansion of the language of rings defined in Section~\ref{sec:eisenstein}.
This additional structure is parameter-freely definable in the valued field and admits natural characterizations (see Remark~\ref{rem:characterise-dagger} and Remark~\ref{rem:induced-str}),
and it is also $\Lring$-definable on the residue field, albeit with parameters.
Our main results can now be stated as follows.
\begin{mtheorem}
Let $(K,v)$ and $(L,w)$ be two finitely ramified henselian valued fields of mixed characteristic $(0,p)$
of initial ramification $e$. 
Then, we have 
\begin{align*}
    \underbrace{(K,v) \equiv (L,w)}_{\text{in }\mathcal{L}_\mathrm{val}} &\Longleftrightarrow
    \underbrace{Kv \equiv Lw}_{\text{in } \Ldagger{e}} \text{ and } \underbrace{vK \equiv wL}_{\text{in }\mathcal{L}_\mathrm{oag}} \\[-0.4cm]
\intertext{and}
     \underbrace{(K,v) \equiv_\exists (L,w)}_{\text{in }\mathcal{L}_\mathrm{val}} &\Longleftrightarrow 
    \underbrace{Kv \equiv_{\exists^+} Lw}_{\text{in } \Ldagger{e}}
\intertext{(where $\equiv_{\exists}$ and $\equiv_{\exists^+}$ denote equality of (positive) existential theories)
and, in case $(K,v) \subseteq (L,w)$,}
     \underbrace{(K,v) \preceq_\exists (L,w)}_{\text{in }\mathcal{L}_\mathrm{val}}
     &\Longleftrightarrow
    \underbrace{Kv \preceq_\exists Lw}_{\text{in }\mathcal{L}_\mathrm{ring}}
     \text{ and }
     \underbrace{vK \preceq_\exists wK}_{\text{in }\mathcal{L}_\mathrm{oag}}.
\end{align*}
\end{mtheorem}
This theorem is proven (as Theorem~\ref{thm:AKE1}, Corollary~\ref{cor:etheory} and Theorem~\ref{thm:exemb}) in Section~\ref{sec:AKE}.
We also give a proof of the analogous version
for $\preceq$ first shown by Ershov and Ziegler.

Previous work on the model theory of finitely ramified henselian valued fields has largely worked not with the residue field as such, but with higher residue rings.
Basarab \cite[Theorem 3.1]{Bas78} proves
that the theory of any finitely ramified henselian valued field is 
determined by the theories of
of its residue rings and value group. More precisely, Basarab shows that
if $(K,v)$ and $(L,w)$
are finitely ramified henselian fields of mixed characteristic $(0,p)$ and the same initial
ramification, then
$$    \underbrace{(K,v) \equiv (L,w)}_{\text{in }\mathcal{L}_\mathrm{val}} \Longleftrightarrow \underbrace{\mathcal{O}_v/p^n \equiv \mathcal{O}_w/p^n}_{\text{in }\mathcal{L}_{\mathrm{ring}}}
\textrm{ for all }n \in \mathbb{N} \textrm{ and } \underbrace{vK \equiv wL}_{\text{in }\mathcal{L}_{\mathrm{oag}}}.$$
For finitely ramified henselian valued fields with perfect residue field, 
Lee and Lee \cite[Theorem 5.2]{LeeLee} show that Basarab's result can be sharpened:
it suffices to consider a single $n=n_0$ depending only on the residue characteristic and the initial ramification, as opposed to all $n \in \mathbb{N}$ simultaneously.
Our results formally imply those of Basarab and Lee--Lee (Remark~\ref{rem:dagger_coincide} and Corollary~\ref{cor:equiv-leelee}), and in fact sharpen them, since we do not need to restrict to perfect residue fields in the phrasing of Lee--Lee.
On the other hand, is not clear how to deduce our results in terms of residue fields from those in terms of higher residue rings.
Our approach also yields further results around stable embeddedness of the residue field (see Section \ref{sec:SE}), which are unattainable using previous techniques.

Unramified valued fields play a natural role in understanding
finitely
ramified ones: every complete $\mathbb{Z}$-valued field is a finite
extension of an unramified complete $\mathbb{Z}$-valued field
with the same residue field \cite[Theorem 11]{Co}, and such an extension is generated by a root of an Eisenstein polynomial.
It now follows from the Ax--Kochen--Ershov Theorem in equicharacteristic zero that every finitely ramified henselian valued
field is -- up to $\mathcal{L}_\mathrm{val}$-elementary equivalence -- a finite extension of an
unramified henselian valued field with the same value group and residue
field.
Our approach is to use Eisenstein polynomials to define a predicate on the residue field sort
of a finitely ramified henselian valued field $(K,v)$ of initial ramification $e$, which is then
added to the $\mathcal{L}_\mathrm{ring}$-language on the residue field to obtain
$\Ldagger{e}$.
\medskip

We now give an overview of the structure of the paper.
In Section~\ref{sec:preliminaries}, we remind the reader of finitely
ramified fields
and
the definability of finitely ramified henselian valuations,
and we give examples
of how relative completeness fails in finitely ramified henselian valued fields. 
In Section \ref{sec:eisenstein},
we then discuss Eisenstein polynomials.
Here we encounter the same bound as Lee and Lee.
We also introduce the $\Ldagger{e}$-structure on the residue field $Kv$ of a finitely ramified henselian valued field $(K,v)$ of initial ramification $e$. 
Section~\ref{sec:embedding-lemmas} provides our toolkit: 
we prove a number of embedding lemmas for the essential case of
$\mathbb{Z}$-valued fields.
These incorporate crucial innovations in the case of inseparable residue field extension,
which allow us to treat imperfect residue fields throughout the work.
We also discuss a number of possible simplifications of
the $\Ldagger{e}$-structure in restricted situations in Remark~\ref{rem:simplified-lang}.

In Section \ref{sec:AKE}, we apply the results from the previous section to prove our main results.
We prove the Ax--Kochen--Ershov principles announced above, and deduce results on a transfer of decidability for full theories (Corollary~\ref{cor:decidability_transfer}) and existential theories (Theorem~\ref{thm:H10}), thus describing precisely the difficulty of Hilbert's 10th Problem for finitely ramified henselian fields.
We also show that the expanded structure on residue fields which we consider throughout is in fact canonical: We characterize it up to positive existential interdefinability (Remark~\ref{rem:characterise-dagger}).
In the final section, we discuss results around stable embeddedness of the value group and residue field.
Such stable embeddedness would follow, as would all our Ax--Kochen--Ershov results, from an appropriate relative quantifier elimination result.
However, we show in Example~\ref{ex:qe-fails} that there is no elimination of valued field quantifiers,
not even in the special case of
unramified henselian valued fields in $\Lval$,
and not even if the structure is enriched by a cross-section or an angular component map.
Nonetheless, we show that
in any finitely ramified henselian valued field $(K,v)$, both $Kv$ and $vK$ are canonically 
stably
embedded (as an $\Ldagger{e}$-structure resp.~an $\Loag$-structure) and orthogonal (Theorem \ref{thm:SE} and Proposition \ref{prop:induced-str}).
In particular, the structure induced on $Kv$ is exactly the $\Ldagger{e}$-structure, which
is in turn definable (with parameters) in the $\Lring$-structure $Kv$.

\section{Preliminaries on finitely ramified valued fields}
\label{sec:preliminaries}
We fix a prime number $p$:
every mixed characteristic valued field we consider in this paper will have residue characteristic $p$.
The aim of this section is to introduce a language and an interpretation,
suitable for finitely ramified fields of residue characteristic $p$ with a fixed initial ramification $e$, 
in order to prove a variety of AKE-type theorems. 
Whenever we consider valued fields as first-order structures in this paper, we consider them 
in the three-sorted language $\mathcal{L}_\mathrm{val}$: 
As usual, this consists of 
a sort $K$ for the field,
a sort $\Gamma\cup\{\infty\}$ for the value group $vK$ together with infinity,
and a sort $k$ for the residue field $Kv$.
The two field sorts are each endowed with the language of rings
$\mathcal{L}_\mathrm{ring} = \{+,\cdot,-,0,1\}$ (where $+$ and $\cdot$ are binary and $-$ is unary),
and the value group sort with $\mathcal{L}_\mathrm{oag} = \{0, +, <,\infty\}$.
In addition, we have symbols for
the valuation map $v \colon K \twoheadrightarrow \Gamma \cup \{ \infty \}$ and the residue map
$\res \colon K \twoheadrightarrow k$, which we interpret as the constant zero map outside the valuation ring $\mathcal{O}_v$.
This language $\Lval$ is close to the ones used in \cite[Section 5.3]{vdD14} and \cite{Belair},
although neither of these sources specify how to interpret $v$ and $\res$ outside $K^\times$ and $\mathcal{O}_v$, respectively.

\begin{definition}
A valued field $(K,v)$ is called \emph{finitely ramified}
if it is of mixed characteristic $(0,p)$ and
the interval $(0,v(p)] \subseteq vK$ is finite.
In this case, we call
$e = |(0,v(p)]|$ the initial ramification of $(K,v)$.
It is \emph{unramified} if moreover $e=1$.
A valuation ring $A$ is finitely ramified (resp., unramified) if it corresponds to a finitely ramified (resp., unramified) valuation $v$ on $\mathrm{Frac}(A)$. 
\end{definition}

The most classical examples of finitely ramified fields are totally ramified extensions
of $(\mathbb{Q},v_p)$ and $(\mathbb{Q}_p, v_p)$, with $e$ being the degree of the field 
extension.

Recall that finitely ramified henselian valuation rings are 
$\mathcal{L}_\mathrm{ring}$-definable without parameters: 
\begin{remark} \label{Robdef}
Given
an initial ramification $e$, choose a natural number $n>e$ which is coprime to $p$ (e.g., $n=ep+1$).
Then,
the valuation ring of any finitely ramified henselian valued field
with
initial ramification $e$ 
is defined by the parameter-free existential $\mathcal{L}_\mathrm{ring}$-formula
\begin{align*} \phi_{p,e}(x) \equiv \exists y: y^n = 1 + p x^{n}.
\end{align*}
Similarly, the maximal ideal of such a valuation ring is defined by the $\mathcal{L}_\mathrm{ring}$-formula
\begin{align*} \psi_{p,e}(x) \equiv \exists y: p y^n = p + x^{n}.
\end{align*}
The latter gives rise to a parameter-free universal definition of the valuation ring.
Both of these are well-known variants of Robinson's formula defining $\mathbb{Z}_p$
in $\mathbb{Q}_p$ \cite[p.~303]{Rob65}.

In particular, the class of fields which carry some finitely ramified henselian valuation
with given initial ramification $e$ is first-order axiomatizable in $\Lring$.
Moreover, as both $\mathcal{O}_v$ and $\mathfrak{m}_v$ are uniformly positively existentially definable without
parameters, we obtain that across this class of fields,
the $\Lval$-structure is positively existentially interpretable in the $\Lring$-structure.
Therefore the $\Lring$-theory of a finitely ramified henselian valued field determines its $\Lval$-theory,
and the same holds for positive existential theories
by the usual rules for interpretations \cite[Theorem 5.3.2 and the following Remark 3]{Hodges_longer}.
In particular, two finitely ramified henselian valued fields are elementarily equivalent in
$\mathcal{L}_\mathrm{val}$ if and only if they are elementarily equivalent in $\mathcal{L}_\mathrm{ring}$.

Lastly, let us note that for an arbitrary valued field $(K,v)$,
the positive existential $\Lval$-theory determines the existential $\Lval$-theory:
This is simply because one can replace inequalities $x \neq y$ in the valued field sort or the residue field sort
by positive existential formulas $\exists z (xz = yz + 1)$ in the appropriate sort,
and in the value group sort there are positive quantifier-free replacements of $\neq$ and $\not<$ using the symbol $<$.%
\footnote{Even if $<$ were not present in our language on the value group sort
  (e.g.\ dropped in favour of $\leq$ as in some other sources),
  the binary relation $\gamma < \gamma'$ would still be positively existentially definable as $\exists x (\res(x) = 0 \wedge \res(x+1) = 0 \wedge \gamma' + v(x) = \gamma)$,
  since the elements $x$ of the valued field with $\res(x) = \res(x+1) = 0$ are precisely those of negative valuation.
  }
Therefore for a finitely ramified henselian valued field,
the existential $\Lring$-theory determines the existential $\Lval$-theory.
\end{remark}

It is clear that fixing the $\mathcal{L}_\mathrm{ring}$-theory of the residue field, the 
$\mathcal{L}_\mathrm{oag}$-theory of the
value group and the initial ramification does not determine a henselian finitely ramified field
up to $\mathcal{L}_\mathrm{val}$-elementary equivalence:
\begin{example} \label{ex:notAKE}
Suppose $p\neq 2$
and let $c \in \mathbb{Z}$ be a quadratic non-residue modulo $p$.
Then the quadratic extensions $\mathbb{Q}_p(\sqrt{p})$ and $\mathbb{Q}_p(\sqrt{cp})$ of $\mathbb{Q}_p$ are distinct, since the elements $p$ and $cp$ of $\mathbb{Q}_p$ lie in distinct square classes by construction.
Now $\mathbb{Q}_p(\sqrt{p})$ and $\mathbb{Q}_p(\sqrt{cp})$ with the unique extensions of the $p$-adic valuation on $\mathbb{Q}_p$ are henselian valued fields with value group isomorphic to $\mathbb{Z}$, residue field $\mathbb{F}_p$ and initial ramification $2$.
As they have different algebraic parts -- indeed, only one of them contains a square root of $p$ --, they
do not have the same existential $\Lring$-theory,
and in particular are not $\Lring$-elementarily equivalent.
(Compare also \cite[Remark 7.4]{AF16} for related observations.)
\end{example}

For finite extensions $K$ of $\mathbb{Q}_p$, it is well-known that the $\Lval$-theory of $K$ (with the unique extension of the $p$-adic valuation on $\mathbb{Q}_p$) is completely determined by the algebraic part $K \cap \overline{\mathbb{Q}}$ \cite[Theorems 3.4 and 5.1]{PR}.
In general, however, even fixing the algebraic part in addition to the
$\mathcal{L}_\mathrm{ring}$-theory of the residue field,
the $\mathcal{L}_\mathrm{oag}$-theory of the value group
and the initial ramification is insufficient to determine the $\Lval$-theory
of the valued field, or even its existential $\Lring$-theory:

\begin{example}\label{ex:hiding}
Again we suppose
$p \neq 2$. 
Consider $F=\mathbb{F}_{p}(t)^{\mathrm{perf}}$ (the perfect hull of $\mathbb{F}_p(t)$),
and let $(K_0, v)$ be the fraction field of the ring of Witt vectors over $F$ with its natural valuation with value group $\mathbb{Z}$ and residue field $F$.
Let $\tau \colon F\longrightarrow K_0$ be the unique multiplicative 
map choosing Teichmüller representatives as in \cite[Theorem 3.3]{AJCohen}.
Let
$\alpha_{1}=\sqrt{p\tau(t)}$,
and
$\alpha_{2}=\sqrt{p\tau(t^{3}+1)}$.
Consider
$K_{1}:=K_0(\alpha_{1})$
and
$K_{2}:=K_0(\alpha_{2})$, each endowed with the unique extension $v_i$ of $v$ to $K_i$.
Then $(K_{1}, v_1)$ and $(K_{2},v_2)$ are both complete (and in particular henselian), 
and they both have the same value group (isomorphic to $\mathbb{Z}$), the same initial ramification 
(namely $2$), and the same residue field $F$.
Moreover, we argue that they also have the same algebraic part, namely $\mathbb{Q}_{p,\mathrm{alg}} = \mathbb{Q}_p \cap \mathbb{Q}^\mathrm{alg}$.
A priori it is clear that the algebraic parts are either $\mathbb{Q}_{p,\mathrm{alg}}$ or ramified quadratic extensions thereof,
since $K_0$ has algebraic part $\mathbb{Q}_{p,\mathrm{alg}}$.
Consider first $K_1$.
If the algebraic part $K_{1,\mathrm{alg}}$ of $K_1$ is a quadratic extension of $\mathbb{Q}_{p,\mathrm{alg}}$ with uniformizer $\pi$,
we have
$p=\tau(a)u \pi^2$
for some $a\in\mathbb{F}_{p}^\times$ and some $u\in K_{1,\mathrm{alg}}$ with $v_1(u-1) > 0$.
Then $u$ has a square root in $K_{1,\mathrm{alg}}$ by Hensel's Lemma,
so $\sqrt{p\tau(a)}\in K_{1,\mathrm{alg}}$ and hence $\sqrt{\tau(at)}\in K_{1}$.
It follows that in the residue field of $K_{1}$ (i.e.~ in $F$), $at$ has a square root --- this is a contradiction. 
Similarly, for $K_2$, if the algebraic part $K_{2,\mathrm{alg}}$ were a proper extension
of $\mathbb{Q}_{p,\mathrm{alg}}$, then for some $a \in \mathbb{F}_p^\times$, $a (t^3+1)$ would
have a square root in $F$, which is again a contradiction.

Nevertheless, $K_1$ and $K_2$ do not have the same existential $\Lval$-theory:
by construction the curve $C \colon Y^2 = X^3 + 1$ has a rational
point with $X$-coordinate $t$ in the residue field $F(\sqrt{t^3+1})$ of $K_{2}(\sqrt{p})$;
however, in the residue field $F(\sqrt{t})$ of $K_{1}(\sqrt{p})$,
all rational points on the curve $C$ of genus $1$ have coordinates in $\mathbb{F}_p$
since $F(\sqrt{t}) = \bigcup_{n \geq 0} \mathbb{F}_p(t^{1/(2p^n)})$ is an increasing union of function fields over $\mathbb{F}_p$
of genus $0$ (see \cite[Lemma 3.2]{Koe02}, as well as the analogous argument in \cite[Example 3.7]{DefinableValuations}).
By Remark \ref{Robdef}, $K_1$ and $K_2$ do not even have the same existential $\Lring$-theory.
\end{example}

\section{Eisenstein polynomials and their traces on residue fields}
\label{sec:eisenstein}

Our approach to the model theory of finitely ramified henselian valued fields is based on understanding arbitrary discrete valuation rings of mixed characteristic as finite ring extensions of unramified ones.
We start with the following well-known fact.
\begin{lemma}\label{lem:eisenstein-generator}
Let $(K,v)$ be a valued field with 
value group $vK \cong \mathbb{Z}$,
and let $(L,w)/(K,v)$ be a finite extension which is totally ramified,
i.e.\ with $[L:K]=(wL : vK)$.
Then any uniformiser of $L$ generates $L$ over $K$, and its minimal polynomial $f$ is an Eisenstein polynomial over $\mathcal{O}_v$.
In other words, $f$ is monic, every non-leading coefficient of $f$ has positive valuation,
and its constant coefficient is a uniformizer of $v$.
\end{lemma}
\begin{proof}
    See \cite[Chapitre I, §6, Proposition 18]{Serre}.
\end{proof}

We collect a few facts on zeroes of Eisenstein polynomials.
We use $v_p$ to denote the $p$-adic valuation on the rational numbers.
For $e \geq 1$, we set 
\begin{align*}
    d(e)
    &=e(1+ v_p(e)) \in \mathbb N.
\end{align*}

\begin{lemma}\label{lem:disc-bound}
    Let $A$ be an unramified discrete valuation ring.
    For any Eisenstein polynomial $f \in A[X]$ of degree $e$ we have $\Delta_f \not\in (p^{d(e)})$, where $\Delta_f \in A$ is the discriminant of $f$.
\end{lemma}
\begin{proof}
    By passing to the completion, we may assume that $A$ is complete.
    Let $K$ be the fraction field of $A$, $\pi$ a zero of $f$ in an algebraic closure of $K$, $L = K(\pi)$, and $B$ the discrete valuation ring of $L$ above $A$ (unique by completeness).
    Since $f$ is an Eisenstein polynomial, $B = A[\pi]$ and $\pi$ is a uniformiser of $B$ \cite[Chapitre I, §6, Proposition 17]{Serre}.
    It follows that, up to a sign, $\Delta_f$ is equal to $N_{B/A}(f'(\pi))$, where $N_{B/A}$ is the norm and $f'$ the formal derivative of $f$ \cite[4.4 Korollar 10]{Bosch}.
    By \cite[Chapitre III, §6, Proposition 13, Remarque]{Serre} we have $v_\pi(f'(\pi)) \leq e - 1 + v_\pi(e) < e + e v_p(e) = d(e)$, where $v_\pi$ and $v_p$ are the normalised valuations of $B$ and $A$, respectively.
    It follows that $v_p(N_{B/A}(f'(\pi))) = v_\pi(f'(\pi)) < d(e)$ using \cite[Chapitre II, §2, Corollaire 4]{Serre}, and so $N_{B/A}(f'(\pi)) \not\in (p^{d(e)})$.
\end{proof}

We deduce that an Eisenstein polynomial has a root
as soon as it has a root modulo $p^{d(e)}$:

\begin{proposition}\label{prop:eisenstein-root-approx}
    Let $A$ be an unramified discrete valuation ring,
	and $B \supseteq A$ a henselian extension.
    An Eisenstein polynomial $f \in A[X]$ of degree $e$ has a zero in $B$ if and only if its reduction mod $p^{d(e)}$ has a zero in $B/p^{d(e)}$.
\end{proposition}
\begin{proof}
    The ``only if'' direction is clear.
    For the converse, the condition that the reduction of $f$ has a zero in $B/p^{d(e)}$ means that there exists $x \in B$ with $f(x) \in (p^{d(e)})$.
    We have $\Delta_f \not\in (p^{d(e)})$ by Lemma \ref{lem:disc-bound},
    and so for the valuation $v$ associated to $B$ we have $v(f(x)) > v(\Delta_f)$.
    Now \cite[Lemma 12]{AxKochen-I} (referring to the ``Hensel--Rychlik property'' defined on \cite[p.~608]{AxKochen-I})
    precisely asserts that under these circumstances $f$ has a zero in $B$,
    using the henselianity assumption.
\end{proof}

We will frequently use the following useful lemma.
Here and in the following, we write $R^{(p^n)}$ for the
set of $p^n$-th powers of a ring $R$.
\begin{lemma}\label{lem:Teich}
Let $(K,v)$ be a finitely ramified valued field
of 
initial ramification $e$
with residue field $k$.
Let $d\geq1$.
For each $\alpha\in k^{(p^{de-1})}$
there is a unique
$a\in(\mathcal{O}_v/p^{d})^{(p^{de-1})}$
with residue $\alpha$.
\end{lemma}
\begin{proof}
First note that, for $x,y\in \mathcal{O}_v$,
if $x\equiv y\pmod{\mathfrak{m}_v^{m}}$ then $x^{p}\equiv y^{p}\pmod{\mathfrak{m}_v^{m+1}}$
(see \cite[Hilfssatz 8]{Teich}).
Let $x,y\in\mathcal{O}_v$
and suppose that
$(x+p^{d}\mathcal{O}_v)^{p^{de-1}}$
and
$(y+p^{d}\mathcal{O}_v)^{p^{de-1}}$
both have residue $\alpha$.
Then 
$x+p^{d}\mathcal{O}_v$
and
$y+p^{d}\mathcal{O}_v$
have the same residue,
i.e.~$x\equiv y\pmod{\mathfrak{m}_v}$.
By the first sentence,
$x^{p^{de-1}}\equiv y^{p^{de-1}}\pmod{\mathfrak{m}_v^{de}}$,
and so
the images of $x^{p^{de-1}}$ and $y^{p^{de-1}}$ in $\mathcal{O}_v/p^{d}$ coincide,
i.e.~$(x+p^{d}\mathcal{O}_v)^{p^{de-1}}=(y+p^{d}\mathcal{O}_v)^{p^{de-1}}$.
\end{proof}

This lemma allows us to make the following definition:
\begin{definition}\label{def:Teich}
For each $d\geq1$,
the map
$\tau \colon k^{(p^{de-1})}\to (\mathcal{O}_{v}/p^{d})^{(p^{de-1})}$
sending each $\alpha$ to the unique $a\in(\mathcal{O}_{v}/p^{d})^{(p^{de-1})}$
with residue $a$
is called the {\em Teichmüller map}.
\end{definition}

We are now ready to define a collection of predicates on the residue field
of a finitely ramified valued field.
These predicates encode information on which polynomials have an approximate root, in particular applying to Eisenstein polynomials over an unramified subring as in Proposition \ref{prop:eisenstein-root-approx}.
This will be used in an essential way in the proof of our main embedding lemma for $\mathbb{Z}$-valued fields, Lemma \ref{lem:emb-imp}.
\begin{definition}\label{def:Omega}
For an $m$-tuple $\underline{b}=(b_j)_{1 \leq j \leq m}$ in some ring
and a multi-index $I = (i_j)_{1 \leq j \leq m}$
we let
$b^{I}:=\prod_{j = 1}^m b_j^{i_j}$ be the $I$-th monomial in the $b_j$.
We denote by $P_{m,n}$ the set of such multi-indices consisting of
indices $i_j < p^n$.
  
Let $(K,v)$ be a finitely ramified valued field of initial ramification $e$ and residue field $k$.
For $d>0$, we write $\res_d \colon \mathcal{O}_v \to \mathcal{O}_v/(p^d)$ for the $d$-th higher residue map.
For $d>0$ and $m\geq 0$,
we define:
\begin{align*}
    \rmOmega_{d,e,m}(K,v)&=
    \left\{(\gamma_{i,j,I}, \beta_l)_{\hspace{-6pt}{\substack{0\leq i<e,\\0<j<d,\\I\in P_{m,de-1},\\ 0\leq l< m}}}\hspace{-6pt}\in k^{e(d-1)p^{(de-1)m}+m}\;
    \left| \;
    \begin{array}{ll}
    \exists \underline{b}\in\mathcal{O}_{v}^{m} \text{ such that }
    \mathrm{res}(\underline{b})=\underline{\beta}
    \text{ and}\\[\medskipamount]
    \text{for }c_{ij}:=\hspace{-6pt}\sum\limits_{{I\in P_{m,de-1}}}\hspace{-6pt}
    \res_{d}(b^{I})\tau(\gamma_{i,j,I}^{p^{de-1}})
    \text{ the}\\ 
    \textrm{polynomial }X^e+\sum\limits_{i=0}^{e-1}(\sum\limits_{j=1}^{d-1} c_{ij}p^j )X^{i}\\[\medskipamount]
    \textrm{has a root in }\mathcal{O}_{v}/p^{d}.
    \end{array}
    \right\}
    \right.
\end{align*}
\end{definition}

We will define our expansion of the language on the residue field using these predicates.
As we want to expand by only one of these, we first need two auxiliary fast-growing functions, whose precise definitions only matter for the technical Lemma \ref{lem:pick-p-indep-to-expand} and Lemma \ref{lem:pick-p-lift-to-expand}.
\begin{definition}\label{def:fast-growing-fns}
  For natural numbers $m,n,g$ define the quantity $M_0(p,m,n,g) \in \mathbb{N}$ by $M_0(p,m,n,0) = m$ and the recursive rule $M_0(p,m,n,g+1) = M_0(p,m+n,n p^{n+m},g)$.
  Define further $M_1(p,n,g,d) \in \mathbb{N}$ by $M_1(p,n,g,0) = 0$ and $M_1(p,n,g,d+1) = M_0(p, M_1(p,n,g,d), n, g)$.
\end{definition}

\begin{definition}[Expanded language $\Ldagger{e}$] \label{def:Ldagger}
  Let $d = d(e)$ and $m = M_1(p,e,de-1,d)$.
We define
$\Ldagger{e}$ to be the expansion of
$\mathcal{L}_{\mathrm{ring}}$ by a
$(e(d-1)p^{(de-1)m}+m)$-ary predicate symbol\footnote{For
  the basic case $e=1$, this means we have a $0$-ary predicate symbol.
  Such symbols are traditionally not considered in model theory
  (although see \cite[Section B.2]{TransseriesBook}), but cause
  no technical problems.}
$\Omega$.
Given a valued field
$(K,v)$ of initial ramification $e$,
we consider the residue field $Kv$ as an $\Ldagger{e}$-structure
by interpreting $\Omega$ as the set $\rmOmega_{d,e,m}(K,v)$.
\end{definition}
Although it is not obvious from the definition,
we will see in Corollary \ref{cor:dagger-defble} that each of the sets $\rmOmega_{d,e,m}(K,v)$, and therefore the entire $\Ldagger{e}$-structure,
is $\Lring$-definable in $Kv$ using parameters.

\begin{remark}\label{rem:dagger_coincide}
  It is clear from the definition that the sets $\Omega_{d,e,m}(K,v)$ only depend on the ring $\mathcal{O}_v/p^d$.
  This has the following consequence:
  If $v_0$ is a proper coarsening of $v$ and $\overline{v}$ the valuation induced on the residue field $Kv_0$, then $\Omega_{d,e,m}(K,v) = \Omega_{d,e,m}(Kv_0, \overline v)$ since $\mathcal{O}_{\overline{v}}/p^d = \mathcal{O}_v/p^d$.
  Therefore the $\Ldagger{e}$-structure induced on the residue field $Kv = (Kv_0)\overline{v}$ is the same for $(K,v)$ and $(Kv_0, \overline v)$.
  Furthermore, since the $\Ldagger{e}$-structure $Kv$ is completely determined by $\mathcal{O}_v/p^{d(e)}$,
  our Ax--Kochen--Ershov style results as announced in the introduction formally imply
  results in terms of residue rings $\mathcal{O}_v/p^n$ (see Corollary \ref{cor:equiv-leelee} and Corollary \ref{cor:ex-equiv-leelee} below)
  in the style of Basarab \cite[Theorem 3.1]{Bas78} and Lee--Lee \cite[Theorem 5.2]{LeeLee}.
\end{remark}

The $\Ldagger{e}$-structure on the residue field is in fact existentially definable in
$\mathcal{L}_\mathrm{val}$.
\begin{lemma}\label{lem:int-imp}
Let $(K,v)$ be a finitely ramified valued field 
with residue field $k$ and initial ramification $e$.
Each set $\rmOmega_{d,e,m}(K,v)$ is existentially definable
without parameters in the
$\mathcal{L}_\mathrm{val}$-structure $(K,v)$
by a formula depending only on $e$ and $p$.
In particular, the
$\Ldagger{e}$-structure is definable on the residue sort of $(K,v)$.
\end{lemma}
\begin{proof}
The Teichmüller map
$\tau:k^{(p^{de-1})}\to (\mathcal{O}_{v}/p^{d})^{(p^{de-1})}$
is existentially definable, and so is
the condition
that $f$ has a root in $\mathcal{O}_{v}/p^{d}$.
\end{proof}

\begin{remark}\label{rem:not-optimal}
  Neither Proposition~\ref{prop:eisenstein-root-approx} nor Lemma~\ref{lem:Teich} above are optimal as stated:
  The ideal $(p^{d(e)})$ in the proposition could be made larger,
  and the exponent $de-1$ in the lemma could be made smaller.
  With more work, it would hence be possible to reduce the arity of the predicates from Definition~\ref{def:Omega}, and hence in the language $\Ldagger{e}$.
  However, the arity in $\Ldagger{e}$ remains very large in any case due to the occurrence of the function $M_1$.
\end{remark}

\section{Embedding Lemmas}
\label{sec:embedding-lemmas}

 We now want to show how to lift homomorphisms of enriched residue fields to homomorphisms
 of finitely ramified valued fields.

\begin{remark}\label{rem:separability}
  In this section, separability of field extensions $l/k$ of characteristic $p$
  and related notions play a significant role.
  For the convenience of the reader, we briefly review this (standard) material,
  mostly following \cite{Mac_SepTrBases}.
  Let $k/k_0$ be an extension fields of characteristic $p$.
  A subset $A \subseteq k$ is \emph{relatively $p$-independent over $k_0$} if $k^{(p)}(k_0 \cup B) \subsetneq k^{(p)}(k_0 \cup A)$ for every $B \subsetneq A$,
  or equivalently, $[k^{(p)}(k_0 \cup A_0) : k^{(p)}(k_0)] = p^{\lvert A_0\rvert}$ for every finite $A_0 \subseteq A$.
  A \emph{relative $p$-basis} of $k/k_0$ is a set $A \subset k$ which is $p$-independent over $k_0$ with $k = k^{(p)}(k_0 \cup A)$,
  or equivalently a set which is maximal among those $p$-independent over $k_0$.
  Without mentioning $k_0$,
  we call $A \subseteq k$ (absolutely) \emph{$p$-independent} or an (absolute) \emph{$p$-basis}
  if it is relatively $p$-independent or a relative $p$-basis over the prime field $k_0 = \mathbb{F}_p$.
  These absolute notions are the most important ones for us.

  A (not necessarily algebraic) field extension $l/k$ is \emph{separable} if it satisfies the following equivalent conditions:
  \begin{enumerate}
  \item Every $p$-independent subset of $k$ remains $p$-independent in $l$.
  \item Some $p$-basis of $k$ remains $p$-independent in $l$.
  \item Every finitely generated subextension $l_0/k$ of $l/k$ has a separating transcendence basis
    (i.e.\ a transcendence basis $B$ such that $l_0/k(B)$ is separable algebraic in the usual sense).
  \end{enumerate}
  (See \cite[Theorem 7 and Theorem 16]{Mac_SepTrBases} for the equivalences;
  note that Mac Lane writes that ``$l/k$ preserves $p$-independence'' instead of the now current terminology ``separable''.)
  This notion of separability agrees with the usual one for algebraic extensions.
\end{remark}

We assume throughout this section and the next that every $\mathbb{Z}$-valued field has mixed characteristic $(0,p)$.
 We start start with the following general lemma for constructing an embedding of an unramified $\mathbb{Z}$-valued field into a complete $\mathbb{Z}$-valued field.
 This generalises \cite[Theorem 6.2]{AJCohen}, where the case of two unramified fields was treated under the assumption that the residue field extension is separable.

\begin{lemma}\label{muensterlemma}
  Let $(K, v)$ and $(L,w)$ be $\mathbb{Z}$-valued fields
 with residue fields $k, l$.
  Assume that $K$ is unramified and $L$ is complete.
  Let $\varphi \colon k \to l$ be a field embedding.
  Let $\underline\beta$ be a (possibly infinite) tuple of $p$-independent elements of $k$, $\underline b$ a lift of $\underline\beta$ in $K$, and $\underline b'$ a lift of $\varphi(\underline\beta)$ in $L$.
  Then there exists an embedding $\Phi \colon (K,v) \to (L,w)$ of valued fields sending $\underline b$ to $\underline b'$, compatible with $\varphi$.
\end{lemma}
\begin{proof}
  By replacing $K$ with its completion, we may as well assume that $K$ is complete, and thus contains $\mathbb{Q}_p$ (identified with the topological closure of the subfield $\mathbb{Q}$).

  We assume first that $k/\mathbb{F}_p$ has a separating transcendence basis
  containing $\underline\beta$.
  By adding more elements to $\underline\beta$ (and to $\underline b$, $\underline b'$),
  we may suppose that $\underline\beta$ is a separating transcendence basis of $k/\mathbb{F}_p$.
  For this case, we reproduce the argument of \cite[Theorem 3]{Mac39b} (although that theorem
  has the superfluous hypothesis that $L$ is also unramified).
  
  Define a field embedding $\Phi_0 \colon \mathbb{Q}_p(\underline b) \to L$ by sending the tuple $\underline b$ to $\underline b'$.
  Considering $\mathbb{Q}_p(\underline b)$ as a valued subfield of $K$, $\Phi_0$ is an embedding of valued fields:
  Indeed, it suffices to check that for every polynomial $f \in \mathbb{Z}_p[\underline X]$, the value of $f(\underline b)$ in $K$ is the same as the value of $f(\underline b')$ in $L$ (where we normalise both valuations to give $p$ the same value) -- but this is true since the residues of $\underline b$ resp.\ $\underline b'$ are the algebraically independent tuples $\underline\beta$ resp.\ $\varphi(\underline\beta)$, so that the value of $f(\underline b)$ is simply the minimum of the values of the coefficients of $f$.

  It is also clear that the embedding $\Phi_0 \colon \mathbb{Q}_p(\underline b) \to L$ is compatible with the map of residue fields $\varphi \colon k \to l$.
  The map $\Phi_0$ extends uniquely (as a valued field embedding) to the closure of $\mathbb{Q}_p(\underline b)$ in $K$, which is identified with the completion $\widehat{\mathbb{Q}_p(\underline b)}$.
  Let us also write $\Phi_0$ for this extension.
  For every finite subextension $K_0$ of $K/\widehat{\mathbb{Q}_p(\underline b)}$, there is a unique embedding $K_0 \hookrightarrow L$ extending $\Phi_0$ and compatible with $\varphi$ (see \cite[Chapitre III, §5, Théorème 3]{Serre}, or apply Hensel's lemma to the image under $\Phi_0$ of the minimal polynomial of a primitive element of $K_0$ over $\widehat{\mathbb{Q}_p(\underline b)}$).
  Hence we have a distinguished embedding of the maximal algebraic subextension of $K/\widehat{\mathbb{Q}_p(\underline b)}$ into $L$.
  Since this maximal algebraic subextension is dense in $K$ as it has the same value group and residue field, we obtain a unique extension to an $\Lval$-embedding $\Phi \colon K \hookrightarrow L$ as desired.

  Let us now consider the general situation.
  Extending $\underline\beta$ (and $\underline b$, $\underline b'$) as above,
  we may assume that $\underline\beta$ is a $p$-basis of $k$.
  Let $(k_i)_{i \in I}$ be the family of all subfields of $k$ admitting a separating transcendence basis over $\mathbb{F}_p$ containing $\underline\beta$.
  For any finitely many elements $x_1, \dotsc, x_n \in k$, the field $\mathbb{F}_p(\underline\beta, x_1, \dotsc, x_n)$ occurs as one of the $k_i$:
  Indeed, $\mathbb{F}_p(\underline\beta, x_1, \dotsc, x_n)$ has a separating transcendence basis over $\mathbb{F}_p(\underline\beta)$ since it is finitely generated and separable, and $\mathbb{F}_p(\underline\beta)$ has the separating transcendence basis $\underline\beta$ over $\mathbb{F}_p$.
  It follows that we can choose an ultrafilter $\mathcal U$ on $I$ such that for every $x \in k$, $\{ i \in I \colon x \in k_i \} \in \mathcal U$.
  With this choice of $\mathcal U$, we have a natural embedding of $k$ into the ultraproduct $k_\infty := \prod_i k_i / \mathcal U$.
  This makes $k_\infty$ a separable extension of $k$, since $p$-independence is preserved.

  For each $i \in I$, we fix an unramified complete $\mathbb{Z}$-valued
  field $(C_i,v_i)$ with residue field $k_i$ and lifts $\underline b_i$ of $\underline \beta$ in $C_i$.
  By the first part, we have an embedding $(C_i,v_i) \hookrightarrow (L,w)$ sending $\underline b_i$ to $\underline b'$ for each $i$, and hence an embedding $(K_\infty,v_\infty) := \prod_i (C_i,v_i) / \mathcal U \hookrightarrow (L,w)^I / \mathcal U =: (L_\infty, w_\infty)$.
  The fields $(K_\infty, v_\infty)$ and $(L_\infty, w_\infty)$ are henselian with residue fields $k_\infty$ resp.\ $l_\infty := l^I / \mathcal U$, and the ultraproduct $\underline b_\infty$ of the tuples $\underline b_i$ provides a lifting of $\underline \beta$ in $K_\infty$.
  Observe that the embedding $K_\infty \hookrightarrow L_\infty$ maps $\underline b_\infty$ to the image of $\underline b'$ under the diagonal embedding $L \hookrightarrow L_\infty$.
  Passing to further ultrapowers if necessary, we may assume that $K_\infty$ and $L_\infty$ are $\aleph_1$-saturated.

  Consider the finest coarsening
  $\mathcal{O}_{v^\infty}^{0}$ (respectively, $\mathcal{O}_{w^\infty}^{0}$)
  of $\mathcal{O}_{v^\infty}$ (respectively, of $\mathcal{O}_{w^\infty}$) of residue characteristic $0$. Note that $v_\infty$ (respectively, $w_\infty$) induces a discrete valuation
  with valuation ring $\mathcal{O}_{\bar{v}_\infty}$ (respectively, $\mathcal{O}_{\bar{w}_\infty}$) and residue field $k_\infty$ (respectively, $l_\infty$) on the residue field of  
  $\mathcal{O}_{v^\infty}^{0}$ (respectively, $\mathcal{O}_{w^\infty}^{0}$).
  By $\aleph_1$-saturation of $K_\infty$ and $L_\infty$, $\mathcal{O}_{\bar{v}_\infty}$ and $\mathcal{O}_{\bar{w}_\infty}$ are complete.%
  \footnote{In fact, $\mathcal{O}_{\bar{v}_\infty}$ is precisely the ultraproduct of the bounded complete metric spaces $\mathcal{O}_{v_i}$,
    and $\mathcal{O}_{\bar{w}_\infty}$ is the ultrapower of $\mathcal{O}_w$,
    both taken with respect to the ultrafilter $\mathcal{U}$ and in the sense of continuous logic.}
  We can identify $\mathcal{O}_w$ with a subring of $\mathcal{O}_{\bar{w}_\infty}$ (using the embedding of $\mathcal{O}_w$ into the valuation ring of $L_\infty$).
  By \cite[Theorem 6.2]{AJCohen}, there is an embedding $\mathcal{O}_v \hookrightarrow \mathcal{O}_{\bar{v}_\infty}$ inducing the embedding $k \hookrightarrow k_\infty$ and sending the given lift $\underline b$ of $\beta$ to $\underline b_\infty$ (or rather its residue in the residue field of $\mathcal{O}_{v_\infty}^{0}$).\footnote{It
    should be stressed that we do not get a ring embedding of $\mathcal{O}_v$ into the valuation ring of $K_\infty$ from the ultraproduct construction.
  Indeed, if $x, y \in K$ are such that the residue field of $\mathbb{Q}(x, y) \subseteq K$ does not have a separating transcendence basis over $\mathbb{F}_p$, then there is no way to embed $\mathbb{Q}(x,y)$ into any of the $C_i$ as a valued field.}

We thus have a composite ring embedding $\mathcal{O}_v \hookrightarrow \mathcal{O}_{\bar{v}_\infty} \hookrightarrow \mathcal{O}_{\bar{w}_\infty}$ inducing the map $k \hookrightarrow k_\infty \hookrightarrow l_\infty$ on residue fields and sending $\underline b$ to (the residue of) $\underline b'$ in the residue field of $\mathcal{O}_{w_\infty}^{0}$.
  It remains to argue that the image of this ring embedding is actually contained in 
  $\mathcal{O}_w \subseteq \mathcal{O}_{\bar{w}_\infty}$.

  To see this, we copy the proof of \cite[Theorem 5.1]{AJCohen} (after Mac Lane).
  By completeness of $\mathcal{O}_w$ it suffices to show that for every $n > 0$ the image of the induced map $\mathcal{O}_{v}/p^n \hookrightarrow \mathcal{O}_{\bar{w}_\infty}/p^n$ is contained in $\mathcal{O}_w/p^n$.
  Choose $N \gg n$.
  Since $k$ is generated by the $\underline\beta$ over $k^{(p^N)}$ \cite[Theorem A1.4~a]{Eisenbud}, the ring $\mathcal{O}_v/p^n$ is generated by the residues of the $\underline b$ and its subset $(\mathcal{O}_v/p^n)^{(p^N)}$ (cf.\ \cite[Lemma 3.5]{AJCohen}).
  Now the images in $\mathcal{O}_{\bar{w}_\infty}/p^n$ of the residues of the $\underline b$ are contained in $\mathcal{O}_w/p^n$ by construction;
  and a $p^N$-th power in $\mathcal{O}_{\bar{w}_\infty}/p^n$ is already determined by its residue in $l_\infty$
  by Lemma~\ref{lem:Teich}
  (using that $\mathcal{O}_{\bar{w}_\infty}$ is finitely ramified, and that $N$ is large), and therefore the image in $\mathcal{O}_{\bar{w}_\infty}/p^n$ of a $p^N$-th power in $\mathcal{O}_{v}/p^n$ lies in $\mathcal{O}_w/p^n$.
\end{proof}

\begin{remark}\label{rem:alternative-muensterlemma}
  We sketch an alternative proof, not using ultraproducts or the weaker embedding result \cite[Theorem 6.2]{AJCohen} for separable residue field extensions.

  By completeness of $(L,v)$, it suffices to construct a compatible family of embeddings $\mathcal{O}_v/p^n \to \mathcal{O}_w/p^n$ for all $n \geq 1$, lifting the map $\varphi$ and sending the residues of the $\underline b$ to the residues of the $\underline b'$.
  In order to construct this embedding, first observe that for large $N \gg n$, the rings $\mathcal{O}_v/p^n$ and $\mathcal{O}_w/p^n$ canonically embed the truncated Witt rings $W_n(k^{(p^N)})$ resp.\ $W_n(l^{(p^N)})$ (essentially by Lemma~\ref{lem:Teich}).
  Secondly, $\mathcal{O}_v/p^n$ is in fact generated over $W_n(k^{(p^N)})$ by the residues of the $\underline b$, after possibly extending $\underline\beta$ to be a full $p$-basis of the residue field $k$.
  The desired embedding of $\mathcal{O}_v/p^n$ into $\mathcal{O}_w/p^n$ is then determined by sending $W_n(k^{(p^N)})$ to $W_n(l^{(p^N)})$ via $\varphi$, and sending the residues of the $\underline b$ to the residues of the $\underline b'$; to verify that this is possible, one only needs to check that all polynomial relations over $W_n(k^{(p^N)})$ satisfied by the $\underline b$ translate to relations over $W_n(l^{(p^N)})$ satisfied by the $\underline b'$, which can be done by hand.

  (This line of reasoning is related to the identification of $\mathcal{O}_v/p^n$ with a ring extension of a truncated Witt ring with explicit generators and relations given by Schoeller \cite[§§ 2--3]{Sch72}.)
\end{remark}

\begin{specialremark}\label{specrem:muensterlemma}
  This remark is not in the published version of this article.

  The proof of Lemma~\ref{muensterlemma} above corrects a small mistake
  in the published version.
  The proof there first considered the special case where $k/\mathbb{F}_p$ has a
  separating transcendence basis, and then claimed that any $p$-basis $\underline\beta$
  of $k$ automatically was such a separating transcendence basis.
  As Junguk Lee made us aware, this is not necessarily the case,
  as for instance the field $\mathbb{F}_p(t_1, t_2, t_3, \dotsc)$
  and its $p$-basis $t_1 + t_2^p, t_2 + t_3^p, t_3 + t_4^p, \dotsc$ show.
  The proof above instead applies the same argument as in the published version
  in the more restrictive special situation
  where $\underline\beta$ can be extended to be a separating transcendence basis of $k/\mathbb{F}_p$,
  which suffices for the later argument in the general situation.

  Meanwhile, a completely different and somewhat simpler proof for
  Lemma~\ref{muensterlemma} was given in
  \cite[Remark~2.28]{Dittmann_formal-smoothness}
  using a formal smoothness argument.
\end{specialremark}

We wish to use Lemma \ref{muensterlemma} to prove an embedding lemma (Lemma \ref{lem:emb-imp} below), asserting that
$\Ldagger{e}$-homomorphisms of residue fields
lift to 
$\mathcal{L}_{\mathrm{val}}$-embeddings of complete $\mathbb{Z}$-valued fields.
To do so in the case of imperfect residue fields, we need two ancillary lemmas concerning representations of elements in terms of a $p$-independent tuple.
These are variants of results which are well-known (see for instance \cite[Theorem A1.4]{Eisenbud} for Lemma \ref{lem:pick-p-indep-to-expand} and \cite[Proposition 3.6]{AJCohen} for Lemma \ref{lem:pick-p-lift-to-expand}), but with a bound for the length of a $p$-independent tuple required.
Recall here the functions $M_0$ and $M_1$ from Definition \ref{def:fast-growing-fns}.

\begin{lemma}\label{lem:pick-p-indep-to-expand}
  Let $k$ be a field of characteristic $p$, $g \geq 0$, and $\beta_1, \dotsc, \beta_m \in k$ a list of $p$-independent elements.
  Given finitely many elements $x_1, \dotsc, x_n \in k$, we can expand the given list to elements $\beta_1, \dotsc, \beta_m, \beta_{m+1}, \dotsc, \beta_{m'}$ which remain $p$-independent, with $m' \leq M_0(p,m,n,g)$, and such that $x_i \in k^{(p^g)}[\underline\beta]$ for all $i$.
\end{lemma}
\begin{proof}
  We use induction on $g$, with the case $g=0$ being clear.
  Therefore suppose that $g > 0$.
  The field extension $k^{(p)}[\beta_1, \dotsc, \beta_m, x_1, \dotsc, x_n]/k^{(p)}[\beta_1, \dotsc, \beta_m]$ has degree at most $p^n$, so we can find a relative $p$-basis $\beta_{m+1}, \dotsc, \beta_{m'}$ for it with $m' \leq m + n$.
  Thus $x_1, \dotsc, x_n \in k^{(p)}[\beta_1, \dotsc, \beta_{m'}]$, so we may write $x_i = \sum_{I \in P_{m',1}} \beta^I y_{i, I}^p$ with elements $y_{i,I} \in k$.
  Now apply the induction hypothesis to the $p$-independent elements $\beta_1, \dotsc, \beta_{m'}$ and the $n \cdot p^{m'}$ many elements $y_{i,I}$,
  using that $M_{0}(p,m,n,g)$ is monotone in $m$ and $n$.
\end{proof}

\begin{lemma}\label{lem:pick-p-lift-to-expand}
  Let $A$ be a complete unramified discrete valuation ring with residue field $k$.
  Let $g, d \geq 0$.
  Given $x_1, \dotsc, x_n \in A$, there exist $p$-independent elements $\beta_1, \dotsc, \beta_m \in k$ with lifts $b_1, \dotsc, b_m \in A$ with $m \leq M_1(p,n,g,d)$ such that
  for each $i = 1, \dotsc, n$ we have
  \[ x_i \equiv \sum_{j=0}^{d-1} \sum_{I \in P_{m,g}} p^j b^I y_{i,j,I}^{p^g} \bmod p^d \]
  for suitable elements $y_{i,j,I} \in A$.
\end{lemma}
\begin{proof}
  We fix $g$, and use induction on $d$.
  There is nothing to be shown for $d = 0$.
  Suppose we have found elements $(\beta_l)_{1 \leq l \leq m}, (b_l)_{1 \leq l \leq m}, (y_{i,j,I})_{1 \leq i \leq n, 0 \leq j \leq d-1, I \in P_{m,g}}$
  as desired for some $d$, with $m \leq M_1(p,n,g,d)$.
  For each $i = 1, \dotsc, n$ write \[ x_i' = p^{-d} (x_i - \sum_{j=0}^{d-1} \sum_{I \in P_{m,g}} p^j b^I y_{i, j,I}^{p^g}) .\]
  By Lemma \ref{lem:pick-p-indep-to-expand} we can expand the list $\beta_1, \dotsc, \beta_m$ to elements
  $\beta_{1},\ldots,\beta_{m},\beta_{m+1},\ldots,\beta_{m'}$,
  with
  $m' \leq M_0(p,m,n,g)$
  such that the residue $x_i' + (p)$ can be written as a $k^{(p^g)}$-linear combination of monomials in the $\beta_l$.
  In other words, $x_i' + (p) = \sum_{I \in P_{m',g}} \beta^I {z'_{i, I}}^{p^g}$ for suitable elements $z'_{i, I} \in k$.
  Choose arbitrary lifts $b_l$ of the newly added $\beta_l$, and a lift $y'_{i,I} \in A$ of each $z'_{i,I}$.
  Then $x_i' \equiv \sum_{I \in P_{m',g}} b^I {y'_{i,I}}^{p^g} \bmod p$, and we obtain the desired representation of $x_i$ mod $p^{d+1}$ by rearranging.
   Finally,
   $m' \leq M_{0}(p,m,n,g)\leq M_0(p,M_{1}(p,n,g,d),n,g)=M_1(p,n,g,d+1)$ holds
   by 
   monotonicity of $M_{0}(p,m,n,g)$ in $m$.
\end{proof}

\begin{lemma}[Embedding lemma for $\mathbb{Z}$-valued fields]\label{lem:emb-imp}
Let $(K,v),(L,w)$ be complete $\mathbb{Z}$-valued fields
with residue fields $k$ and $l$ respectively, each
of initial ramification $e$. Then,
every $\Ldagger{e}$-homomorphism
$k \to l$
is induced by an $\mathcal{L}_{\mathrm{val}}$-embedding
$(K,v)\to (L,w)$.
Moreover
every $\Ldagger{e}$-isomorphism
$k \to l$
is induced by an $\mathcal{L}_{\mathrm{val}}$-isomorphism
$(K_,v)\to(L,w)$.
\end{lemma}
\begin{proof}
Let $\varphi \colon k \rightarrow l$ be an $\Ldagger{e}$-homomorphism.
By \cite[Theorem 11]{Co},
there exists a subring
$\mathcal{O}_0\subseteq \mathcal{O}_{v}$
which is a complete unramified valuation ring (i.e.\ a Cohen ring) with residue field $k$.

Let $\pi \in K$ be a uniformizer of $v$.
By Lemma \ref{lem:eisenstein-generator}, the minimal polynomial of $\pi$ over $K_0$ is an Eisenstein polynomial
$F 
=X^{e}+\sum_{i<e}c_{i}X^{i}
\in \mathcal{O}_{0}[X]$.
Let $d=d(e)$ and consider
$f:=\res_{d}(F)\in (\mathcal{O}_0/p^d) [X]$.
We now show how $f$ encodes an element of $\Omega_{d,e,m}(K,v)$ for $m = M_1(p,e,de-1,d)$.
By Lemma \ref{lem:pick-p-lift-to-expand}, there exist $m' \leq M_1(p,e,de-1,d)$, $p$-independent elements $\beta_1, \dotsc, \beta_{m'} \in k$ and lifts $b_1, \dotsc, b_{m'} \in \mathcal{O}_0$ of the $\beta_i$ such that for each of the coefficients $c_i$ we have
\[ c_i \equiv \sum_{j=0}^{d-1} p^j \sum_{I \in P_{m',de-1}} b^I y_{i,j,I}^{p^{de-1}} \bmod p^d \]
with suitable $y_{i,j,I} \in \mathcal{O}_0$.
Note that $c_i \in p\mathcal{O}_0$ since $F$ is Eisenstein,
and so $y_{i,0,I} \equiv 0 \bmod p$
since the $\beta_l$ are $p$-independent,
and hence $y_{i,0,I}^{p^{de-1}} \equiv 0 \bmod p^{d}$ 
for all $i, I$.
Define $\beta_l = 0$ for $m' < l \leq m$, and $\gamma_{i,j,I} = 0$ for multi-indices $I \in P_{m, de-1} \setminus P_{m',de-1}$.
Let $\gamma_{i,j,I} \in k$ be the residue of $y_{i,j,I}$ for all $i,j,I$.
The Teichmüller map $\tau \colon k^{(p^{de-1})} \to (\mathcal{O}_v/(p^d))^{(p^{de-1})}$ sends $\gamma_{i,j,I}^{p^{de-1}}$ to
the residue of $y_{i,j,I}^{p^{de-1}}$ mod $p^d$,
see Definition~\ref{def:Teich},
and so we have $((\gamma_{i,j,I})_{ijI}, (\beta_l)_l) \in \Omega_{d,e,m}(K,v)$ by construction.

Since $\varphi$ is an
$\Ldagger{e}$-homomorphism, we have
$(\varphi(\gamma_{i,j,I}),\varphi(\beta_{l}))\in\Omega_{d,e,m}(L,w)$.
By definition of $\Omega_{d,e,m}(L,w)$, there is some 
$\underline{b'}\in\mathcal{O}_{w}^{m}$
with
$\res(\underline{b'})=\varphi(\underline{\beta})$
such that the polynomial $g \in (\mathcal{O}_w/p^d)[X]$ with coefficients determined by $\varphi(\gamma_{i,j,I})$ and
$\mathrm{res}_d((\underline{b'})^I)$ has a root in $\mathcal{O}_w/p^d$.
 Applying Lemma \ref{muensterlemma}, we obtain an embedding $\Phi_0 \colon K_0 \to L$
 of valued fields sending $b_1, \dotsc, b_{m'}$ to $b'_1, \dotsc, b'_{m'}$.

By construction, the polynomial $G = \Phi_0(F) \in \Phi_0(K_0)[X] \subseteq L[X]$ is an Eisenstein polynomial over $\Phi_0(\mathcal{O}_0)$ since $F$ is an Eisenstein polynomial over $\mathcal{O}_0$, and $G$ has residue $\mathrm{res}_d(G)=g$.
As $g$ has a root in $\mathcal{O}_w/p^d$ and $d=d(e)$, $G$ has a root in $L$
by Proposition~\ref{prop:eisenstein-root-approx}.
Thus we may extend $\Phi_0$ to an
$\mathcal{L}_{\mathrm{ring}}$-embedding
$\Phi \colon K\to L$, which still lifts $\varphi$.
By Remark \ref{Robdef}, $\Phi$ automatically respects the valuations,
and so gives an $\mathcal{L}_\mathrm{val}$-embedding.

For the final claim:
If $\varphi$ is an isomorphism, then $L$ and its unramified subfield $\Phi(K_0)$ both have the same residue field, and therefore $[L : \Phi(K_0)] = e = [K : K_0]$.
This shows that $\Phi$ must be surjective, and therefore an isomorphism of valued fields.
\end{proof}

\begin{remark}\label{rem:lee-lee}
  Since the $\Ldagger{e}$-structure on the residue field $Kv$ is completely determined by $\mathcal{O}_v/p^{d(e)}$ (Remark~\ref{rem:dagger_coincide}), the lemma yields as an immediate consequence some results related to questions studied in \cite{LeeLee}.
  For instance, we obtain that $(K,v) \cong (L,w)$ if and only if $\mathcal{O}_v/p^{d(e)} \cong \mathcal{O}_w/p^{d(e)}$ (cf.~\cite[Question 1.3]{LeeLee}).
  Crucially we do not need the hypothesis, imposed throughout \cite{LeeLee}, that the residue fields are perfect.
\end{remark}

\begin{remark}\label{rem:simplified-lang}
In the remainder of this article, we will use Lemma~\ref{lem:emb-imp} as a black box.
We will also not use the precise definition of the $\Ldagger{e}$-structure on the residue field again, as it was only used in order to prove the lemma.

  We note some settings in which the $\Ldagger{e}$-structure on the residue field can be significantly simplified.
  As a consequence, all of our main theorems remain true in these settings with the simplified structure, since they only depend on Lemma~\ref{lem:emb-imp}.
  The key idea for each of these simplifications is that the $\Ldagger{e}$-structure is needed in the proof of Lemma~\ref{lem:emb-imp} to capture the minimal polynomial $F$ of a uniformizer $\pi$ of $(K,v)$ over a Cohen subring, or rather a suitable reduction of $F$.
  Under suitable hypotheses, we obtain restrictions on the shape of $F$, at least after a wise choice of uniformizer.
    \begin{enumerate}
    \item If $(K,v)$ is unramified (i.e. $e=1$ and thus $d(e)=1$)),
      then we simply have $\Omega_{1,1,0}(K,v) = (Kv)^0$.
      In particular, we can dispense with $\Ldagger{e}$ and use $\Lring$ instead.
    \item If the initial ramification $e$ under consideration is coprime to $p$, so that we only consider tamely ramified fields, then the uniformizer $\pi$ in the proof of Lemma~\ref{lem:emb-imp} can always be chosen to have minimal polynomial $F = X^e - pa$, where $a \in \mathcal{O}_0^\times$ is a unit of the Cohen subring $\mathcal{O}_0$.
      (See for instance \cite[Chapter II, §5, Proposition 12]{Lang_ANT}, where the standing hypothesis that the residue field is perfect is not necessary, or alternatively
      \cite[Kapitel II, Satz 7.7]{Neukirch} and its proof.)
      By Hensel's Lemma, any element of a henselian valued field with residue characteristic $p \nmid e$ whose residue is an $e$-th power is itself an $e$-th power.
      We can deduce from this that a polynomial $F$ of the given form has a root in $(K,v)$ if it has a root modulo $p^2$.
      It therefore follows that instead of the predicate symbol $\Omega$ of large arity, for Lemma~\ref{lem:emb-imp} it suffices to use a unary predicate symbol $\Omega^{\mathrm{tame}}$ interpreted as
      \[ \Omega^{\mathrm{tame}}_e(K,v) =
           \left\{\alpha\in Kv 
			\;\left|\;
			\begin{array}{ll}
			\exists a\in\mathcal{O}_{v}^{\times} \text{ with }
			\mathrm{res}(a)=\alpha,
			\text{ and }
			f=X^e - pa \\ \textrm{has a }
			\textrm{root in }\mathcal{O}_{v}/p^2
			\end{array}
			\right\}
                        \right.
                        .
      \]
      This set is always an $e$-th power class in the residue field $Kv$.                
    \item If we restrict to finitely ramified fields of initial ramification $e$ whose residue field has finite degree of imperfection bounded by $m \in \mathbb{N}$, then we can work with the set $\Omega_{d(e),e,m}$, as opposed to using $\Omega_{d(e),e,M_1(p,e,de-1,d)}$.
      For low values of $m$, this significantly lowers the arity of the predicate considered since the function $M_1$ grows very quickly.
    \item As a special case of the previous point, if we work only with perfect residue fields, it suffices to use a predicate for the set
      \begin{align*} 
         \rmOmega_{d(e),e,0}(K,v)&=
    \left\{(\gamma_{i,j})_{{\substack{0\leq i<e,\\0<j<d(e)}}} \in k^{e(d(e)-1)}\;
    \left|\;
    \begin{array}{ll}
    \textrm{for }
    c_{i}:=\sum_{j=1}^{d(e)-1}\tau(\gamma_{i,j}^{d(e)e-1})p^{j},
    \text{ the polynomial }\\[\medskipamount]
    f=X^e + \sum_{i=0}^{e-1}c_{i}X^{i}
    \text{ has a root in }\mathcal{O}_{v}/p^{d(e)} 
    \end{array}
    \right\}
    \right. 
    \end{align*}
    (where $\tau$ denotes the Teichmüller map from Definition~\ref{def:Teich}).
    \end{enumerate}
\end{remark}

\begin{remark}
In this remark we briefly remove our standing assumption that $p$ is fixed:
we define a language
$\mathcal{L}_{\mathrm{unif}, e}$
in which we may treat all residue characteristics $p>0$ uniformly, 
and which only depends on $e$.
Recall that the quantities $d(e)$ and $m$ in fact also depend on $p$,
so we write $d(e,p)=e(1+v_{p}(e))$
and
$m(e,p)=M_{1}(p,e,d(e,p)e-1,d(e,p))$.
In $\mathcal{L}_{\mathrm{unif}, e}$ we add a $(e(d(e,p)-1)p^{(d(e,p)e - 1)m(e,p)}+m(e,p))$-ary predicate
$\Omega_{p}$
for all prime numbers $p$ to the language $\mathcal{L}_\mathrm{ring}$. 
Then given any valued field $(K,v)$
of mixed characteristic
and initial ramification $e$,
we may expand the residue field
$k=Kv$ to an
$\mathcal{L}_{\mathrm{unif}, e}$-structure
by interpreting
$\Omega_{p}$
as the set
\begin{align*}
   \left\{\begin{array}{ll}
   \Omega_{d(e,p),e,m(e,p)}(K,v)
   &\text{when $p=\mathrm{char}(Kv)$}\\
   \emptyset
   &\text{when $p\neq\mathrm{char}(Kv)$}.
   \end{array}\right.
\end{align*}
It can be checked that
Lemma~\ref{lem:emb-imp}
goes through with $\mathcal{L}_{\mathrm{unif}, e}$
in place of $\Ldagger{e}$.
We consider this construction to have no advantage over the setting with a fixed residue characteristic $p$
and hence we do not pursue this point of view in the rest of the paper.
\end{remark}
    
Lemma \ref{lem:emb-imp} will be sufficient to prove an Ax--Kochen--Ershov principle for existential equivalence in the next section,
and (although we will take a different route below)
it could also be used to establish the corresponding principle for elementary equivalence --
for instance, one could copy the proof of the Ax--Kochen--Ershov principle for
elementary equivalence of unramified henselian valued fields \cite[Theorem 8.5 and Corollary 8.3]{AJCohen},
using our Lemma~\ref{lem:emb-imp} instead of the corresponding result \cite[Corollary 6.6]{AJCohen} in the unramified case.

To handle elementary extensions and existential closedness, however, we will need embedding lemmas over substructures. In fact, here we do not require any additional structure on the residue
field.
We start with the following self-strengthening of Lemma \ref{muensterlemma}, with proof similar to that of 
  \cite[Corollary 6.4]{AJCohen}. 

\begin{lemma}\label{muensterlemma-subfield}
  Let $(K_0, v_0)$ be an unramified $\mathbb{Z}$-valued field with residue field $k_{0}$, and $(K, v)$ an unramified extension of $(K_0, v_0)$ with residue field $k$ such that the extension $k/k_0$ is separable.
  Let $(L,w)$ be a complete $\mathbb{Z}$-valued field.
  Let $\Phi_0 \colon K_0 \hookrightarrow L$ be an embedding of valued fields.
  Let $\varphi \colon k \to l$ be a field embedding which on $k_0 \subseteq k$ agrees with the map induced by $\Phi_0$.
  Let $\underline\beta$ be a (possibly infinite) tuple of elements of $k$ which are relatively $p$-independent over $k_0$,
  $\underline b$ a lift of $\underline\beta$ in $K$, and $\underline b'$ a lift of $\varphi(\underline\beta)$ in $L$.
  Then there exists an embedding $\Phi \colon (K,v) \to (L,w)$ of valued fields extending $\Phi_0$, sending $\underline b$ to $\underline b'$, and compatible with $\varphi$.
\end{lemma}
\begin{proof}
We denote by $\varphi_{0}:k_{0}\hookrightarrow l$ the embedding of residue fields induced by $\Phi_{0}$.
Let $\underline{\beta}_0$ be a $p$-basis of $k_{0}$.
Note that $\underline{\beta}_{0}\cup\underline{\beta}$ is $p$-independent in $k$.
Let $\underline{b}_{0}$ be a lift of $\underline{\beta}_{0}$ in $K_{0}$.
Then $\underline{b}_{0}\cup\underline{b}$ is a lift of $\underline{\beta}_{0}\cup\underline{\beta}$ in $K$.
Moreover, $\Phi_{0}(\underline{b}_{0})\cup\underline{b}'$ is a lift of $\varphi_{0}(\underline{\beta}_{0})\cup\varphi(\underline{\beta})$.
By Lemma~\ref{muensterlemma},
there exists an embedding $\Phi:(K,v)\rightarrow L$  sending $\underline{b}_{0}\cup\underline{b}$ to $\Phi_{0}(\underline{b}_{0})\cup\underline{b}'$
which is compatible with $\varphi$.

What is left to show is that $\Phi$ extends $\Phi_0$ on all of $K_0$.
Note that $\Phi_0$ and the restriction $\Phi|_{K_0}$ are both isomorphisms onto their image,
coinciding on $\underline{b_0}$.
Therefore the images of $\Phi_{0}$ and $\Phi|_{K_{0}}$ coincide by the argument of \cite[Theorem 5.1]{AJCohen} already used at the end of the proof of Lemma \ref{muensterlemma}.
By the uniqueness assertion in \cite[Theorem 6.2]{AJCohen},
we conclude that $\Phi_{0}$ and $\Phi|_{K_{0}}$ are equal.
\end{proof}
Comparing with \cite[Corollary 6.4]{AJCohen}, here the assumption that $l/\varphi(K_0v_0)$ is separable has been dropped.
On the other hand, the assumption that $k/K_0v_0$ is separable is necessary, as already Mac Lane points out in \cite[Examples I and II after Theorem 12]{Mac39b}; we will use this phenomenon later for an example of a failure of quantifier elimination in Example \ref{ex:qe-fails}.

We now deduce another embedding lemma that is at the heart of relative model completeness and relative
existential completeness.

\begin{lemma}[Embeddings over valued subfields]\label{lem:emb-rank1-subfield-silly}
  Let $(K_0, v_0), (K,v), (L,w)$ be $\mathbb{Z}$-valued fields, each of initial ramification $e$, where $K_0$ is a valued subfield of $K$.
  Let furthermore an embedding $\Phi_0 \colon K_0 \hookrightarrow L$ of valued fields be given, with $\varphi_0 \colon K_0v_0 \hookrightarrow Lw$ the induced residue map.
  Assume that $(L,w)$ is complete and $Kv/K_0v_0$ is separable.
  Then every $\mathcal{L}_{\mathrm{ring}}$-embedding $\varphi \colon Kv\to Lw$ extending $\varphi_0$ is induced by an $\mathcal{L}_{\mathrm{val}}$-embedding $\Phi \colon K \to L$ extending $\Phi_0$.
  If $K$ is complete and $\varphi$ is an isomorphism, then so is $\Phi$.
\end{lemma}
\begin{proof}
  We may suppose that $(K,v)$ is complete (by replacing it with its completion), and then also that $(K_0,v_0)$ is complete.
  Let $C_0 \subseteq K_0$ be a complete unramified subfield with residue field $K_0v_0$.
  There exists a valued extension field $C/C_0$ which is complete and unramified with residue field $Kv$, and by Lemma \ref{muensterlemma-subfield} we can identify $C$ over $C_0$ with a subfield of $K$.

  We have $[K : C] = [K_0 : C_0] = e$, and so $K$ is the compositum of $K_0$ and $C$; furthermore, $K_0$ and $C$ are linearly disjoint over $C_0$.
  By Lemma \ref{muensterlemma-subfield} once more, we can extend the embedding $C_0 \hookrightarrow \Phi_0(C_0) \subseteq L$ to an embedding $C \hookrightarrow L$ respecting $\varphi$.
  Since the embeddings $C \hookrightarrow L$ and $K_0 \hookrightarrow L$ agree on $C_0$ by construction, we get an embedding defined on the compositum $K = C K_0$, which has the desired properties.
  If $\varphi$ is an isomorphism, then $\Phi(K)$ is a complete subfield of $L$ with the same residue field and absolute ramification index, hence $\Phi(K) = L$.
\end{proof}

We have now established a number of embedding lemmas for $\mathbb{Z}$-valued fields.
In order to be able to later achieve strong results related to stable embeddedness,
we now establish embedding lemmas for general finitely ramified fields.
The following one is modelled on \cite[Proposition 10.1]{AJCohen}, which handles the unramified case.
Compared to loc.\ cit., we drop a separability assumption due to our stronger statement of Lemma \ref{muensterlemma},
and also shorten the proof by applying a general embedding lemma for valued fields due to Basarab \cite{Basarab}.
\begin{lemma}\label{lem:emb-gen-subfield}
  Let $(K,v)$ and $(L,w)$ be two extensions of a valued field $(K_0, v_0)$, and suppose that all three fields are henselian and finitely ramified of the same ramification index.
  Assume that $v_0K_0$ is pure in $vK$, i.e.\ $vK/v_0K_0$ is torsion-free, and that $Kv/K_0v_0$ is separable.
  Moreover, assume that $(K_0, v_0)$ and $(K,v)$ are $\aleph_1$-saturated, and that $(L,w)$ is $\lvert K\rvert^+$-saturated.

  Then every pair of an $\Loag$-embedding $\varphi_\Gamma \colon vK \hookrightarrow wL$ over $v_0K_0$ and an $\Lring$-embedding $Kv \hookrightarrow Lw$ over $K_0v_0$ is induced by an embedding $K \hookrightarrow L$ over $K_0$.
\end{lemma}
\begin{proof}
  By our saturation assumption, $(K_0, v_0)$ carries a cross-section $v_0K_0 \to K_0^\times$ of the valuation \cite[Lemma 7.9]{vdD14}, and $(K,v)$ carries a cross-section extending this \cite[Lemma 7.10]{vdD14}.
  As for $(L,w)$, we cannot expect a cross-section extending the cross-section of $K_0$ since $v_0K_0$ is not necessarily pure in $wL$;
  however, we at least have a partial cross-section $\varphi_\Gamma(vK) \to L^\times$ extending the one of $K_0$:
  Indeed, this follows like in \cite[Lemma 7.10]{vdD14} from the fact that $v_0K_0$ is a direct summand of $\varphi_\Gamma(vK)$.

  Let us write $v_0^0$, $v^0$, $w^0$ for the finest proper coarsenings of the three valuations; their residue fields have characteristic $0$, and by saturation they are complete with respect to the rank-$1$ valuations $\overline{v_0}$, $\overline{v}$, $\overline{w}$ induced by the original valuations.
  
  We have an exact sequence
  \[ 1 \to \mathcal{O}_v^\times/1 + \mathfrak{m}_{v^0} \to K^\times/1 + \mathfrak{m}_{v^0} \to vK \to 0, \]
  with a splitting provided by the cross-section.
  We may identify the first term of the exact sequence with $\mathcal{O}_{\overline v}^\times$ (via the residue map $\mathcal{O}_v^\times \to \mathcal{O}_{\overline v}^\times$).
  We have analogous exact sequences for $(K_0, v_0)$ and $(L,w)$.
  Since the cross-sections of $(K_0, v_0)$ and $(K, v)$ are compatible, this means that we have compatible isomorphisms
  \[ K^\times/1+\mathfrak{m}_{v^0} \cong \mathcal{O}_{\overline{v}}^\times \times vK, \quad
    K_0^\times/1+\mathfrak{m}_{v_0^0} \cong \mathcal{O}_{\overline{v_0}}^\times \times v_0K_0.\]
  
  By Lemma \ref{lem:emb-rank1-subfield-silly}, we have an embedding $\Phi_0 \colon (Kv^0, \overline v) \hookrightarrow (Lw^0, \overline w)$ above $(K_0v_0^0, \overline{v_0})$ inducing the given embedding of residue fields $Kv \hookrightarrow Lw$.
  This $\Phi_0$ in particular induces an embedding $\mathcal{O}_{\overline v}^\times \hookrightarrow \mathcal{O}_{\overline w}^\times$ over $\mathcal{O}_{\overline{v_0}}^\times$.
  Taking the composite map $\mathcal{O}_{\overline v}^\times \hookrightarrow \mathcal{O}_{\overline w}^\times \hookrightarrow L^\times / 1 + \mathfrak{m}_{w^0}$ together with the map $vK \to \varphi_\Gamma(vK) \to L^\times / 1 + \mathfrak{m}_{w^0}$ obtained from the partial cross-section of $(L,w)$,
  we obtain an embedding \[ \psi \colon K^\times/1+\mathfrak{m}_{v^0} \cong \mathcal{O}_{\overline v}^\times \times vK \hookrightarrow L^\times/1+\mathfrak{m}_{w^0}\] over $K_0^\times/1+\mathfrak{m}_{v_0^0}$.
  
  Let $\pi \in K_0$ be the uniformizer given as the image of the smallest positive element of $v_0K_0$ under the cross-section $v_0K_0 \to K_0^\times$.
  By construction, $\psi$ preserves the class of $\pi$ in $K_0^\times/1+\mathfrak{m}_{v_0^0}$.
  Since $(Kv^0)^\times$ is generated by $\mathcal{O}_{\overline v}^\times$ and the residue class of $\pi$, it follows that $\psi$ sends \[(Kv^0)^\times \cong \mathcal{O}_{v^0}^\times/1+\mathfrak{m}_{v^0} \subseteq K^\times/1+\mathfrak{m}_{v^0}\] to \[(Lw^0)^\times \cong \mathcal{O}_{w^0}^\times/1+\mathfrak{m}_{w^0} \subseteq L^\times/1+\mathfrak{m}_{w^0} ,\]
  and in fact agrees there with $\Phi_0$, since we also have $\Phi_0(\pi) = \pi$.
  In other words, $\psi$ fits in the following diagram with exact rows, above the corresponding diagram for $K_0$:
  \[ \xymatrix{
      1 \ar[r] & (Kv^0)^\times \ar[r] \ar^{\Phi_0}[d] & K^\times/1+\mathfrak{m}_{v^0} \ar[r] \ar^\psi[d] & v^0K \ar[r] \ar@{-->}[d] & 0 \\
      1 \ar[r] & (Lw^0)^\times \ar[r] & L^\times/1+\mathfrak{m}_{w^0} \ar[r] & w^0L \ar[r] & 0
  }\]
  Here, by construction of $\psi$,
  the induced dashed arrow $v^0 K \to w^0 L$ is the map induced by $\varphi_\Gamma \colon v K \to w L$.

  By \cite[Theorem 2.1]{Basarab} (cf.\ also the presentation in \cite[Théorème 4.2]{Belair}), the embedding $\psi$ is induced by an embedding $K \hookrightarrow L$ over $K_0$.
  (The hypothesis there that $(L,w_0)$ is $\lvert K\rvert$-pseudocomplete is satisfied since $(L,w)$ is $\vert K\rvert^+$-saturated; see for instance \cite[Proof of Theorem 4.2]{Basarab}.)
  By construction, this embedding will then induce the given embeddings $vK \hookrightarrow wL$ and $Kv \hookrightarrow Lw$.
\end{proof}

From the preceding relative embedding lemma over the subfield $K_0$
we obtain the following absolute statement.
\begin{lemma}\label{lem:emb-gen-abs}
  Let $(K, v)$ and $(L,w)$ be two henselian finitely ramified valued field of the same ramification index $e$.
  Suppose $(L,w)$ is $\lvert K\rvert^+$-saturated.

  Then every pair of an $\Loag$-embedding $\varphi_\Gamma \colon vK \hookrightarrow wL$
  sending the minimal positive element of $vK$ to the minimal positive element of $wL$
  and an $\Ldagger{e}$-homomorphism $\varphi_k \colon Kv \hookrightarrow Lw$
  is induced by an embedding $(K,v) \hookrightarrow (L,w)$.
\end{lemma}
\begin{proof}
  Let $(K^\ast, v^\ast)$, $(L^\ast, w^\ast)$, $\varphi_\Gamma^\ast \colon v^\ast K^\ast \hookrightarrow w^\ast L^\ast$ and $\varphi_k^\ast \colon K^\ast v^\ast \hookrightarrow L^\ast w^\ast$
  be obtained as elementary extensions of the corresponding unstarred objects
  in such a way that $(K^\ast, v^\ast)$ and $(L^\ast, w^\ast)$ are $\aleph_1$-saturated.
  Let $v^{\ast 0}$ and $w^{\ast 0}$ be the finest proper coarsenings of $v^\ast$ and $w^\ast$.
  By the saturation, the residue fields $K^\ast v^{\ast 0}$ and $L^\ast w^{\ast 0}$ are complete with respect to
  the rank-$1$ valuations $\overline{v^\ast}$ and $\overline{w^\ast}$ induced by $v^\ast$ and $w^\ast$.
  By Lemma~\ref{lem:emb-imp}, the $\Ldagger{e}$-homomorphism $\varphi_k^\ast$
  is induced by an embedding $(K^\ast v^{\ast 0}, \overline{v^\ast}) \hookrightarrow (L^\ast w^{\ast 0}, \overline{w^\ast})$.

  Furthermore, the henselian valued fields $(K^\ast, v^{\ast 0})$ and $(L^\ast, w^{\ast 0})$ of residue characteristic $0$
  have sections $K^\ast v^{\ast 0} \hookrightarrow K^\ast$ and $L^\ast w^{\ast 0} \hookrightarrow L^\ast$ (see for instance \cite[Lemma 2.3]{AF16}).
  We may therefore see $K_0^\ast := K^\ast v^{\ast 0}$ as a subfield of both $K^\ast$ and $L^\ast$.
  The restriction of the valuations $v^\ast$ and $w^\ast$ to $K_0^\ast$ must be the valuation $v_0^\ast := \overline{v^\ast}$,
  since this is the unique finitely ramified valuation on $K_0^\ast$.

  We are now in the situation of Lemma \ref{lem:emb-gen-subfield}, except for the saturation assumptions:
  Namely, $\varphi_\Gamma^\ast \colon v^\ast K^\ast \hookrightarrow w^\ast L^\ast$ is an embedding preserving the group $v_0^\ast K_0^\ast \cong \mathbb{Z}$ by the assumption
  that $\varphi_\Gamma^\ast$ preserves the minimal positive element.
  The embedding $\varphi_k^\ast \colon K^\ast v^\ast \hookrightarrow L^\ast w^\ast$ is over $K_0^\ast v_0^\ast = K^\ast v^\ast$ by construction --
  indeed, this is how we chose the embedding $K_0^\ast v_0^\ast \hookrightarrow L^\ast w^{\ast 0}\overline{w^\ast} = L^\ast w^\ast$ in the first place.
  The extension $K_0^\ast v_0^\ast/K^\ast v^\ast$ is trivial, hence separable,
  and the group $v^\ast K^\ast/v_0^\ast K_0^\ast = v^{\ast 0} K^\ast$ is torsion-free.
  After replacing all of $(K^\ast,v^\ast)$, $(L^\ast,w^\ast)$, $(K_0^\ast, v_0^\ast)$, and the appropriate maps between them,
  by further elementary extensions,
  we can arrange for enough saturation for Lemma \ref{lem:emb-gen-subfield} to be applied.

  Let us return to the original valued fields $(K,v)$ and $(L,w)$ with maps $\varphi_\Gamma$ und $\varphi_k$.
  We have shown that there exist elementary extensions $(K^\ast, v^\ast)$, $(L^\ast, w^\ast)$ with maps $\varphi^\ast_\Gamma$ und $\varphi^\ast_k$,
  all extending the respective unstarred versions,
  and an embedding $(K^\ast, v^\ast) \hookrightarrow (L^\ast, w^\ast)$ inducing $\varphi^\ast_\Gamma$ und $\varphi^\ast_k$.
  In particular, the resulting composite embedding $(K,v) \hookrightarrow (K^\ast, v^\ast) \hookrightarrow (L^\ast, w^\ast)$
  induces on the value groups the map $vK \hookrightarrow wL \hookrightarrow w^\ast L^\ast$ and on the residue fields the map $Kv \hookrightarrow Lw \hookrightarrow L^\ast w^\ast$,
  coming from $\varphi_\Gamma$ and $\varphi_k$.

  Considering both $(K,v)$ and $(L,w)$ as $\Lval(vK, Kv)$-structures
  (where for $(L,w)$, this structure comes from $\varphi_\Gamma$ and $\varphi_k$),
  this means that $(K,v)$ can be embedded into an elementary extension of $(L,w)$ in this expanded language.
  The $\lvert K \rvert^+$-saturation of $(L,w)$ then implies that $(K,v)$ can be $\Lval(vK, Kv)$-embedded into
  $(L,w)$ itself.
  (For instance, use \cite[Theorem 8.1.7]{Hodges_longer} with the set of formulas expressing
  the quantifier-free relations of elements of $(K,v)$.)
\end{proof}

\section{Ax-Kochen--Ershov principles}
\label{sec:AKE}
In this section, we prove a range of Ax--Kochen--Ershov principles for finitely ramified fields.

\subsection{Relative model completeness and completeness}

Our embedding results from the previous section imply relative model completeness
(which was already shown independently by Ershov \cite[Theorem 4.3.4]{ErshovMult} and
Ziegler \cite[Satz V.5~I)~iii)]{ZieglerDiss}):
\begin{theorem}[Ax--Kochen--Ershov Principle 1: relative model completeness]\label{thm:RMC}
Let $(K,v)\subseteq(L,w)$ be two finitely ramified henselian valued fields
of the same initial ramification $e$.
Suppose the inclusions
$Kv \subseteq Lw$
and
$vK\subseteq wL$
are elementary in
$\mathcal{L}_{\mathrm{ring}}$
and
$\mathcal{L}_{\mathrm{oag}}$, respectively.
Then 
$(K,v)\preceq(L,w)$.
\end{theorem}
\begin{proof}
  Completely analogous to \cite[Theorem 9.2]{AJCohen}, using Lemma \ref{lem:emb-rank1-subfield-silly} instead of \cite[Corollary 6.4]{AJCohen}.
\end{proof}

We can use this to establish relative completeness, arguably the archetypal result of Ax--Kochen--Ershov type.
\begin{theorem}[Ax--Kochen--Ershov Principle 2: relative completeness]\label{thm:AKE1}
Let $(K,v)$ and $(L,w)$ be two finitely ramified henselian valued fields 
of initial ramification $e$. 
Then, we have
\begin{align*}
    \underbrace{(K,v) \equiv (L,w)}_{\text{in }\mathcal{L}_\mathrm{val}} \Longleftrightarrow \underbrace{Kv \equiv Lw}_{\text{in }\Ldagger{e}} \text{ and } \underbrace{vK \equiv wL}_{\text{in }\mathcal{L}_\mathrm{oag}}
\end{align*}
\end{theorem}

In fact, we prove the result in the following stronger version,
allowing a comparison of types of elements of the residue fields and value groups.
\begin{proposition}\label{prop:AKE-equiv-plus}
  Let $(K,v)$ and $(L,w)$ be two finitely ramified henselian valued fields 
of initial ramification $e$.
  Let $\underline a$, $\underline a'$ be tuples in the residue fields $Kv$ resp.\ $Lw$ of the same length,
  and let $\underline b$, $\underline b'$ be tuples in the value groups $vK$ and $wL$ of the same length.
  Then $$\tp_{(K,v)}^{\Lval}(\underline a, \underline b) = \tp_{(L,w)}^{\Lval}(\underline a', \underline b') \Longleftrightarrow
  \tp_{Kv}^{\Ldagger{e}}(\underline a) = \tp_{Lw}^{\Ldagger{e}}(\underline a')\textrm{ and }
  \tp_{vK}^{\Loag}(\underline b) = \tp_{wL}^{\Loag}(\underline b').$$ 
\end{proposition}
\begin{proof}
  Since $\tp_{Kv}^{\Ldagger{e}}(\underline a)$ and $\tp_{vK}^{\Loag}(\underline b)$ are
  encoded in $\tp_{(K,v)}^{\Lval}(\underline a, \underline b)$,
  and analogously for $\underline a', \underline b'$, the forward direction is clear.
  Assume conversely that $\tp_{Kv}^{\Ldagger{e}}(\underline a) = \tp_{Lw}^{\Ldagger{e}}(\underline a')$
  and $\tp_{vK}^{\Loag}(\underline b) = \tp_{wL}^{\Loag}(\underline b')$.
  Replacing $(L,w)$ by a sufficiently saturated elementary extension,
  we may suppose that we have an elementary $\Ldagger{e}$-embedding $Kv \hookrightarrow Lw$ sending $\underline a$ to $\underline a'$,
  and similarly an elementary $\Loag$-embedding $vK \hookrightarrow wL$ sending $\underline b$ to $\underline b'$.
  By Lemma~\ref{lem:emb-gen-abs}, these elementary embeddings come from an $\Lval$-embedding $(K,v) \hookrightarrow (L,w)$.
  This embedding is itself elementary by Theorem~\ref{thm:RMC}.
  The equality of $\Lval$-types follows.
\end{proof}

\begin{proof}[Proof of Theorem \ref{thm:AKE1}]
  This is just Proposition \ref{prop:AKE-equiv-plus}
  for empty tuples $\underline a$, $\underline a'$, $\underline b$ and $\underline b'$.
\end{proof}

As usual, Theorem~\ref{thm:AKE1} yields a transfer of decidability.
\begin{corollary}\label{cor:decidability_transfer}
Let $(K,v)$ be finitely ramified henselian of initial ramification $e$.
The $\mathcal{L}_\mathrm{val}$-theory of $(K,v)$ is decidable if and only if
both the $\mathcal{L}_\mathrm{oag}$-theory of $vK$ and the 
$\Ldagger{e}$-theory of $Kv$ are decidable.
\end{corollary}
\begin{proof}
  This is standard.
  The ``only if'' is immediate because both the $\Loag$-structure $vK$ and the $\Ldagger{e}$-structure $Kv$ are interpretable in the $\Lval$-structure $(K,v)$ (see \cite[Remark 4 to Theorem 5.3.2]{Hodges_longer}).
  For the converse, assume that both $vK$ and $Kv$ are decidable.
  Then the theory $T$ of finitely ramified henselian valued fields $(L,w)$ of the same initial ramification as $(K,v)$ satisfying $Lw \equiv Kv$ (in the language $\Ldagger{e}$) and $wL \equiv vK$ has an obvious computable axiomatization.
  By Theorem~\ref{thm:AKE1}, $T$ is complete, and therefore the $\Lval$-theory of $(K,v)$ is precisely the set of consequences of $T$, which is decidable by a proof calculus.
\end{proof}

We also note the following corollary in the spirit of \cite[Theorem 5.2]{LeeLee},
but without a perfectness assumption on the residue fields as imposed there.
\begin{corollary}\label{cor:equiv-leelee}
  Let $(K,v)$ and $(L,w)$ be two finitely ramified henselian valued fields of the same initial ramification $e$.
  Let $d = d(e) = e (1 + v_p(e)) \in \mathbb{N}$ as in Section~\ref{sec:eisenstein}.
  Then
  \begin{align*}
    \underbrace{(K,v) \equiv (L,w)}_{\text{in }\mathcal{L}_\mathrm{val}} \Longleftrightarrow
    \underbrace{\mathcal{O}_v/p^{d} \equiv \mathcal{O}_w/p^{d}}_{\text{in }\Lring} \text{ and }
    \underbrace{vK \equiv wL}_{\text{in }\mathcal{L}_\mathrm{oag}} .
  \end{align*}
  
\end{corollary}
\begin{proof}
  The forward direction is clear since the higher residue rings $\mathcal{O}_v/p^d$ and $\mathcal{O}_w/p^d$ are
  interpretable in $(K,v)$ and $(L,w)$, respectively, by the same formulas.
  For the backward direction, we observe that $\mathcal{O}_v/p^d$ and $\mathcal{O}_w/p^d$ being
  $\Lring$-elementarily equivalent forces $Kv$ and $Lw$ to $\Ldagger{e}$-elementarily equivalent
  (Remark~\ref{rem:dagger_coincide}).
  Therefore the statement follows from Theorem~\ref{thm:AKE1}.
\end{proof}

\subsection{Relative existential completeness and Hilbert's 10th problem}

In the following proposition, we use $\exists\textrm{-}\tp_{(K,v)}^{\Lval}(\underline a)$ to denote
the existential type of $\underline{a}$ in $(K,v)$, and write $\exists^{+}\textrm{-}\tp_{Kv}^{\Ldagger{e}}(\underline a)$ for the positive existential type.

\begin{proposition} \label{prop:etypes}
Let $(K,v)$, $(L,w)$ be finitely ramified henselian valued fields 
of the same initial ramification $e$.
Let $\underline a$, $\underline a'$ be tuples in the residue fields $Kv$ resp.\ $Lw$ of the same length.
Then $$\exists\textrm{-}\tp_{(K,v)}^{\Lval}(\underline a) \subseteq \exists\textrm{-}\tp_{(L,w)}^{\Lval}(\underline a') \Longleftrightarrow \exists^{+}\textrm{-}\tp_{Kv}^{\Ldagger{e}}(\underline a)\subseteq \exists^{+}\textrm{-}\tp_{Lw}^{\Ldagger{e}}(\underline a').$$
\end{proposition}
\begin{proof}
By Lemma \ref{lem:int-imp}, the $\Ldagger{e}$-structure on the residue fields $Kv$ and $Lw$ is
existentially $\Lval$-definable, uniformly in the fields $(K,v)$ and $(L,w)$.
The direction from left to right is an immediate consequence.

Let us prove the converse direction.
Passing if necessarily to elementary extensions,
we may assume that $(K,v)$ is $\aleph_{1}$-saturated
and that $(L,w)$ is $|K|^{+}$-saturated.
By saturation, there is an $\Ldagger{e}$-homomorphism
$Kv \to Lw$
that carries $\underline{a}$ to $\underline{a}'$.
Let $v_0$ and $w_0$ be the finest proper coarsenings of $v$ respectively $w$,
and $\bar{v}$ and $\bar{w}$ the valuations induced on the residue fields
$Kv_0$ and $Lw_0$.
By saturation,
$(Kv_0, \bar{v})$ and $(Lw_0, \bar{w})$ are complete $\mathbb{Z}$-valued
fields of initial ramification $e$.
As these induce the same 
$\Ldagger{e}$-structure on $Kv$ and $Lw$ as $(K,v)$ and $(L,w)$ 
by Remark~\ref{rem:dagger_coincide}, the
homomorphism $Kv \to Lw$ is induced by an embedding
$Kv_0 \to Lw_0$
 by Lemma \ref{lem:emb-imp}.
This embedding shows that
$\exists\textrm{-}\tp_{(Kv_{0},\bar{v})}^{\Lval}(\underline{a})\subseteq\exists\textrm{-}\tp^{\Lval}_{(Lw_{0},\bar{w})}(\underline{a}')$.
By a standard argument
(e.g.~\cite[Lemma 2.3]{AF16}, noting that in characteristic zero every field extension is separable)
there is a section $f:Kv_{0}\rightarrow K$ of the residue map of $v_{0}$.
Then $v_{0}$ corresponds to a $f(Kv_{0})$-rational $f(Kv_{0})$-place.
Since $Kv_{0}$ is large and perfect,
$f:Kv_{0}\to K$ is an existentially closed $\Lring$-embedding,
by \cite[Theorem 17]{Kuh04}.
Therefore $K$ can be embedded into an elementary extension of the field $Kv_0$, compatibly with $f$.
By Remark~\ref{Robdef}, this embedding respects the valuations,
and so $f$ gives an existentially closed $\Lval$-embedding $(Kv_0, \bar{v}) \to (K,v)$.
Furthermore, notice that $f$ induces the identity on the common residue field $Kv_0\bar{v} = Kv$
since $f$ is a section of $v_0$.
Thus
$\exists\textrm{-}\tp_{(K,v)}^{\mathcal{L}_{\mathrm{val}}}(\underline{a})=\exists\textrm{-}\tp_{(Kv_{0},\bar{v})}^{\mathcal{L}_{\mathrm{val}}}(\underline{a})$.
The same argument also yields
$\exists\textrm{-}\tp_{(Lw_{0},\bar{w})}^{\mathcal{L}_{\mathrm{val}}}(\underline{a}')=\exists\textrm{-}\tp_{(L,w)}^{\mathcal{L}_{\mathrm{val}}}(\underline{a}')$.
\end{proof}

\begin{corollary}[Ax--Kochen--Ershov Principle 3: relative existential completeness]\label{cor:etheory}
Let $(K,v)$, $(L,w)$ be finitely ramified henselian valued fields 
of the same initial ramification $e$.
 Then:
 \[ \Th_\exists^\Lval(K,v) \subseteq \Th_\exists^\Lval(L,w) 
     \Longleftrightarrow 
     \Th_{\exists^+}^{\Ldagger{e}}(Kv) \subseteq \Th_{\exists^+}^{\Ldagger{e}}(Lw)  \]    
\end{corollary}
\begin{proof}
This immediately follows from Proposition \ref{prop:etypes}, choosing empty tuples $\underline{a}$
and $\underline{a'}$.
\end{proof}

Note that for notational symmetry, one may replace $\Th_\exists^\Lval(K,v)$ by
$\Th_{\exists^+}^\Lval(K,v)$ (or even by $\Th_{\exists^+}^{\Lring}(K)$) in the statement of the corollary,
since these theories determine one another in a straightforward way (Remark \ref{Robdef}).
However, it is not possible to replace $\Th_{\exists^+}^{\Ldagger{e}}(Kv)$ by $\Th_{\exists}^{\Ldagger{e}}(Kv)$:

\begin{example}\label{ex:existential-no-plus}
Let $p \neq 2$, $k= \mathbb F_p^\mathrm{alg}$ and 
$l = \mathbb F_p^\mathrm{alg}(t)^\mathrm{perf}$.
We write $W[k]$, $W[l]$ and $W[l^{\mathrm{alg}}]$ for the ring of Witt vectors over $k$ resp.\ $l$ and $l^{\mathrm{alg}}$.
Consider the finitely ramified henselian valued fields
$K = \operatorname{Frac}(W[k])(\sqrt{p})$ and $L = \operatorname{Frac}(W[l])(\sqrt p)$, together with the natural valuations $v$ resp.\ $w$ given as the unique prolongations of the valuations with valuation ring $W[k]$ resp.\ $W[l]$.
Both $(K,v)$ and $(L,w)$ have initial ramification $e=2$.
Since $l^\mathrm{alg} \succ k$, we have $\operatorname{Frac}(W[l^{\mathrm{alg}}]) \succ \operatorname{Frac}(W[k])$ \cite[Theorem 7.2]{vdD14} and therefore $K \subseteq L \subseteq \operatorname{Frac}(W[l^\mathrm{alg}])(\sqrt p)$ shows $\Th_\exists(K) = \Th_\exists(L)$.
By Remark \ref{Robdef}, even the existential $\Lval$-theories of $(K,v)$ and $(L,w)$ agree.

On the other hand, the existential $\Ldagger{e}$-theories $\Th_\exists(k)$ and $\Th_\exists(l)$ are not equal:
Indeed, every quadratic Eisenstein polynomial over $W(k)$ has a root in $K$ since $K$ is the unique quadratic extension of $W(k)$, whereas this is not the case for $L/W(l)$ since $W(l)$ has more than one ramified quadratic extension.
For instance, the polynomial $X^2 - pt$ has no root in $L$.
There is an existential (but not positive existential) $\Ldagger{e}$-sentence expressing that there exists a quadratic polynomial of the form $X^2 - pa$, with $a$ a unit, which does not have a zero in the valuation ring (equivalently by Proposition \ref{prop:eisenstein-root-approx}, in the valuation ring modulo $p^{d(2)}$) -- namely, this asserts that there exists a certain tuple which does not lie in the set given by the $\Omega$-predicate.
This existential sentence holds in $l$, but not in $k$.
\end{example}

We can reformulate Proposition~\ref{prop:etypes} as follows:
\begin{corollary}\label{cor:eformulas}
  Let $e \geq 1$.
  For every existential $\Lval$-formula $\varphi(\underline x)$, where $\underline x$ is a tuple of variables of the residue field sort, there exists a positive existential $\Ldagger{e}$-formula $\tau(\underline x)$ such that the following holds:
  For every finitely ramified henselian valued field $(K,v)$ of initial ramification $e$, and every tuple of elements $\underline a$ in $Kv$, we have $(K,v) \models \varphi(\underline a)$ if and only if $Kv \models \tau(\underline a)$.
\end{corollary}
\begin{proof}
  Let $T$ be the theory of finitely ramified henselian valued fields of initial ramification $e$.
  For every positive existential $\Ldagger{e}$-formula $\tau(\underline x)$, let $\widetilde\tau(\underline x)$ be the existential $\Lval$-formula asserting that $\tau$ holds in the residue field sort, where every occurrence of the $\Omega$-predicate is replaced by a defining existential formula (Lemma \ref{lem:int-imp}).

  We wish to show that for the given formula $\varphi(\underline x)$, there exists a positive existential $\tau(\underline x)$ with $T \models \varphi \leftrightarrow \widetilde\tau$.
  By \cite[Lemma B.9.2]{TransseriesBook} it suffices to show that given two complete $T$-types $p(\underline x)$, $q(\underline x)$ with $\varphi \in p$, $\neg\varphi \in q$ there exists a positive existential $\Ldagger{e}$-formula $\tau(\underline x)$ with $\widetilde\tau \in p$, $\neg\widetilde\tau \in q$.
  This is a restatement of Proposition \ref{prop:etypes}.
\end{proof}

\begin{remark}\label{rem:characterise-dagger}
  Corollary~\ref{cor:eformulas} shows that our choice of the additional structure put on the residue field (Definitions \ref{def:Omega} and \ref{def:Ldagger}) was far from arbitrary:
  our predicate symbol $\Omega$ is interpreted by a set which is existentially $\Lval$-definable, and any other existentially definable subset of the residue field may be defined positively existentially in terms of $\Omega$.
  This property characterises the $\Ldagger{e}$-structure up to positive existential interdefinability.

  We mentioned in Remark~\ref{rem:simplified-lang} that in certain situations, all of our results hold with other languages replacing $\Ldagger{e}$.
  It follows that these alternative structures on the residue field are positively existentially interdefinable with the $\Ldagger{e}$-structure.
  For instance, for $e$ not divisible by $p$, the $\Ldagger{e}$-structure on residue fields is positively existentially interdefinable with the $\Lring \cup \{\Omega^{\mathrm{tame}}\}$-structure.
\end{remark}

We deduce from Corollary~\ref{cor:etheory} a result on existential decidability, i.e.\ the solvability of Hilbert's 10th Problem.
We state this more precisely than the analogous result for full theories, Corollary~\ref{cor:decidability_transfer}.
Recall here that for theories $T_1$ and $T_2$ in finite languages $\mathcal{L}_1$ resp.\ $\mathcal{L}_2$ we say that $T_1$ is \emph{many-one reducible} to $T_2$ if there is an effective procedure associating to each $\mathcal{L}_1$-sentence $\varphi$ an $\mathcal{L}_2$-sentence $\varphi^\ast$ such that $\varphi \in T_1$ if and only if $\varphi^\ast \in T_2$.
(See for instance \cite[Definition 1.6.8]{Soare}.)
\begin{theorem} \label{thm:H10}
Let $(K,v)$ be a finitely ramified henselian valued field of initial ramification $e$.
Then $\Th_\exists(K,v)$ and $\Th_{\exists^+}^{\Ldagger{e}}(Kv)$ are many-one reducible
to one another.
In particular, one is decidable if and only if the other is.
\end{theorem}
\begin{proof}
The many-one reducibility of $\Th_{\exists^+}^{\Ldagger{e}}(Kv)$ to $\Th_\exists(K,v)$ follows from the fact (Lemma~\ref{lem:int-imp}) that the $\Ldagger{e}$-structure $Kv$ is existentially interpretable in $(K,v)$ (see for instance \cite[Remark 4 to Theorem 5.3.2]{Hodges_longer}.

For the converse direction, let $e$ be the initial ramification index of $(K,v)$, and let $\varphi$ be an existential $\Lval$-sentence.
By Corollary~\ref{cor:eformulas} there exists a positive existential $\Ldagger{e}$-sentence $\psi$ such that for any finitely ramified henselian valued field $(L,w)$ 
 with initial ramification $e$ we have $(L,w) \models \varphi$ if and only if $Lw \models \psi$.
Since the theory of all such $(L,w)$ has a natural computable axiomatization, an exhaustive search using a proof calculus means that in fact such $\psi$ can be effectively found given $\varphi$.
The mapping assigning to each $\varphi$ a suitable $\psi$ is the required reduction.
\end{proof}

By Remark~\ref{Robdef}, we can replace the existential $\Lval$-theory of $(K,v)$ by the existential $\Lring$-theory of $K$ in the theorem above.
On the other hand, it is essential to use the existential $\Ldagger{e}$-theory of the residue field instead of the existential $\Lring$-theory, as the following example shows.
\begin{example} \label{ex:hidex}
  By \cite[Theorem 1.1]{Ditt22}, there is a complete $\mathbb{Z}$-valued field $(K,v)$ such that the existential $\Lring$-theory of $Kv$ is decidable and the algebraic part $K\cap \mathbb{Q}^\mathrm{alg}$ is decidable, but such that the 
existential $\mathcal{L}_\mathrm{val}$-theory of $(K,v)$ is not decidable.
\end{example}

Finally, we give a consequence of our Ax--Kochen--Ershov principle (Corollary \ref{cor:etheory})
in terms of higher residue rings.
\begin{corollary}\label{cor:ex-equiv-leelee}
  Let $(K,v)$ and $(L,w)$ be two finitely ramified henselian valued fields of the same initial ramification $e$.
  Let $d = d(e) = e (1 + v_p(e)) \in \mathbb{N}$ as in Section~\ref{sec:eisenstein}.
  Then
  \begin{align*}
    \underbrace{(K,v) \equiv_\exists (L,w)}_{\text{in }\mathcal{L}_\mathrm{val}} \Longleftrightarrow
    \underbrace{\mathcal{O}_v/p^{d} \equiv_{\exists^+} \mathcal{O}_w/p^{d}}_{\text{in }\Lring}
  \end{align*}
  
\end{corollary}
\begin{proof}
  This follows from Corollary \ref{cor:etheory} much like Corollary \ref{cor:equiv-leelee} was deduced from Theorem \ref{thm:AKE1}.
  Indeed, $\mathcal{O}_v/p^d$ is quantifier-freely interpretable in $(K,v)$, so $\Th_\exists(K,v)$ controls $\Th_{\exists^+}(\mathcal{O}_v/p^d)$,
  and analogously for $(L,w)$.
  This yields the forward direction.
  For the backward direction,
  the $\Ldagger{e}$-structure $Kv$ is positively existentially interpretable in $\mathcal{O}_v/p^d$:
  indeed, the ring $Kv$ is $\mathcal{O}_v/p^d$ modulo its maximal ideal, which is the set of elements $x \in \mathcal{O}_v/p^d$
  whose $ed$-th power vanishes,
  and it is clear from the definition of the predicates $\Omega_{d,e,m}$ on $Kv$ that
  they are positively existentially definable in $\mathcal{O}_v/p^d$
  (cf.\ Remark \ref{rem:dagger_coincide} and Lemma \ref{lem:int-imp}).
  Since the analogous statements hold for $(L,w)$,
  the condition $\mathcal{O}_v/p^d \equiv_{\exists^+} \mathcal{O}_w^d$ implies that $Kv \equiv_{\exists+} Lw$ as $\Ldagger{e}$-structures,
  and so the statement follows from Corollary \ref{cor:etheory}.
\end{proof}

\subsection{Existential closedness}

We now come to our Ax--Kochen--Ershov principle for existential closedness.

\begin{theorem}[Ax--Kochen--Ershov Principle 4: existential closedness]\label{thm:exemb}
Let $(K,v)\subseteq(L,w)$ be two finitely ramified henselian valued fields
of the same initial ramification.
Suppose the inclusions
$Kv\subseteq Lw$
and
$vK\subseteq wL$
are existentially closed in
$\mathcal{L}_{\mathrm{ring}}$
and
$\mathcal{L}_{\mathrm{oag}}$, respectively.
Then 
$(K,v)\preceq_\exists (L,w)$.
\end{theorem}
\begin{proof}
This is analogous to \cite[Theorem 10.2]{AJCohen}:
taking ultrapowers if necessary, we may assume that both $(K,v)$ and $(L,w)$ are $\aleph_{1}$-saturated. Consider an $|L|^{+}$-saturated elementary extension 
$(K^{*},v^{*})$ 
of $(K,v)$.
As $Kv\preceq_{\exists}Lw$, there is an embedding
$\varphi_{k}:Lw\hookrightarrow K^{*}v^{*}$ over $Kv$.
Furthermore, $Kv \preceq_\exists Lw$ also ensures that a $p$-independent subset of $Kv$
remains $p$-independent in $Lw$,
so $Lw/Kv$ is separable.
Since $vK\preceq_{\exists}wL$,
there is an embedding $\varphi_{\Gamma}:wL\hookrightarrow v^{*}K^{*}$
over $vK$.
Note also that $wL/vK$ is torsion-free.
Thus, Lemma~\ref{lem:emb-gen-subfield} gives an embedding
$(L,w)\hookrightarrow(K^{*},v^{*})$ over $K$ inducing $\varphi_{k}$ and $\varphi_{\Gamma}$.
We conclude $(K,v)\preceq_{\exists}(L,w)$.
\end{proof}

  Note that \cite[Theorem 10.2]{AJCohen} assumes that $(K,v)$ and $(L,w)$ have the same finite
  degree of imperfection.
  In Theorem \ref{thm:exemb}, there is no need to restrict to finite degree of imperfection,
  because we have one fewer separability hypothesis in our embedding lemma. Thus, 
  we have a generalization of \cite[Theorem 10.2]{AJCohen} even in the unramified case.

\section{Induced structure on residue field and value group}
\label{sec:SE}
  
In this section we show that in a finitely ramified henselian valued field
of initial ramification $e$ considered in the language $\mathcal{L}_\mathrm{val}$, 
the structure induced on $Kv$ is exactly the $\Ldagger{e}$-structure,
and the structure induced on $vK$ is exactly the $\Loag$-structure.
Moreover, we show that $Kv$ and $vK$ are stably embedded and orthogonal, and
that the $\Ldagger{e}$-structure is $\Lring$-definable (with parameters) on $Kv$.

These results would follow from a relative quantifier elimination. We now give an example to show that
such a relative quantifier elimination fails, even in the unramified setting:
\begin{example}\label{ex:qe-fails}
    Consider $p = 2$.
    Let $(K,v)$ be a complete $\mathbb{Z}$-valued unramified henselian field with residue field $\mathbb{F}_2(s,t)$ where $s$ and $t$ are algebraically independent. Let $b_s$ and $b_t$ be
    lifts of $s$ and $t$ in $K$. We now consider two degree $4$ extensions of $K$ (each with the
    unique extension of $v$):
    $F:= K(\sqrt{b_s + b_t},\sqrt{b_s})$ and $L:=K(\sqrt{b_s}, \sqrt{b_t})$.
    Note that $L$ does not contain a square root of $b_s + b_t$:
    indeed, the quadratic subextensions of $L/K$ are precisely $K(\sqrt{b_s})$, $K(\sqrt{b_t})$ and $K(\sqrt{b_sb_t})$, and none of them has residue field $\mathbb{F}_2(s, t, \sqrt{s+t})$, as would be required for the field $K(\sqrt{b_s+b_t})$.
    (This is related to \cite[p.~435, Example I]{Mac39b}.)
    Therefore the $\Lring$-type of $b_s+b_t$ is different in $L$ and in $F$, and a fortiori the same holds for the $\Lval$-types.

    On the other hand, we have $\sqrt{s+t} = \sqrt{s} + \sqrt{t}$ since we are in characteristic $2$, and therefore there is an obvious isomorphism $Fv \cong \mathbb{F}_2(\sqrt{s+t}, \sqrt{t}) = \mathbb{F}_2(\sqrt{s}, \sqrt{t}) \cong  Lv$ fixing $Kv$.
    This implies that for any $\Lval$-formula $\varphi(x)$ without quantifiers over the valued field sort, we have $F \models \varphi(b_s+b_t)$ if and only if $L \models \varphi(b_s+b_t)$.
    The same holds if we allow $\varphi$ to involve a symbol for the cross-section $vK = \mathbb{Z} \to K$, $n \mapsto p^n$, or a symbol for the angular component map naturally defined in terms of this cross-section (see \cite[beginning of Section 5.4]{vdD14}).

    Therefore the theory of unramified henselian fields of mixed characteristic $(0,2)$, even in the language with a symbol for a cross-section or an angular component map, does not eliminate valued field quantifiers.
    In fact, since $F$ and $L$ are isomorphic (abstractly, i.e.\ not as extensions of $K$) \cite[Corollary 1 to Theorem 8]{Mac39b}, even the theory $\Th(F) = \Th(L)$ with these symbols does not eliminate valued field quantifiers.
    This is in contrast to the situation in equicharacteristic zero, see \cite[Corollary 5.24]{vdD14}. 
\end{example}

We adopt the following notation, as in the proof of Corollary~\ref{cor:eformulas}:
For an $\Ldagger{e}$-formula $\varphi_1(\underline x)$, we let $\widetilde\varphi_1(\underline x)$ be the $\Lval$-formula asserting that $\varphi_1$ holds in the residue field sort, where every occurrence of the $\Omega$-predicate is replaced by a defining formula (Lemma~\ref{lem:int-imp}).
Likewise, for an $\Loag$-formula $\varphi_2(\underline x)$, we let $\widetilde\varphi_2(\underline x)$ be the natural $\Lval$-formula asserting that $\varphi_2$ holds in the value group sort.
Note that both $\widetilde\varphi_1$ and $\widetilde\varphi_2$ make sense also if $\varphi_1$ and $\varphi_2$ involve parameters.
\begin{theorem} \label{thm:SE}
    Let $(K,v)$ be a finitely ramified henselian valued field.
    Then any $\Lval(K)$-formula $\varphi(\underline x, \underline y)$, with $\underline x$ a tuple of variables of the residue field sort and $\underline y$ a tuple of variables of the value group sort, is equivalent in $K$ to a boolean combination of formulas of the form $\widetilde\varphi_1(\underline x)$ or of the form $\widetilde\varphi_2(\underline y)$, with $\varphi_1(\underline x)$ a $\Lring(Kv)$-formula and $\varphi_2(\underline y)$ a $\mathcal{L}_{\mathrm{oag}}(vK)$-formula.
\end{theorem}
\begin{proof}
  By a standard compactness argument, see \cite[Corollary B.9.3]{TransseriesBook}, it suffices to show that two complete $K$-types $p(\underline x, \underline y)$ and $q(\underline x, \underline y)$ are equal if their restriction to $\Lring(Kv)$-types in the variables $\underline x$ are equal and their restriction to $\mathcal{L}_{\mathrm{oag}}(vK)$-types in the variables $\underline y$ are equal.

  Let therefore $(K_1,v_1), (K_2,v_2)$ be elementary extensions of $(K,v)$,
  let $\underline a, \underline a'$ be tuples in $K_1v_1$ resp.~$K_2v_2$ with $\tp(\underline a/Kv) = \tp(\underline a'/Kv)$,
  and let $\underline b, \underline b'$ be tuples in $v_1K_1$ resp.~$v_2K_2$ with $\tp(\underline b/vK) = \tp(\underline b'/vK)$.
  We wish to show that $\tp(\underline a, \underline b/K) = \tp(\underline a', \underline b'/K)$.

  By passing to ultrapowers of $K, K_1$ and $K_2$, we may assume that $K$ and $K_1$ are $\aleph_1$-saturated.
  By replacing $(K_2,v_2)$ by a further elementary extension if necessary, we may assume that $(K_2,v_2)$ is $\lvert K_1\rvert^+$-saturated.
  By saturation, there exists an elementary embedding $\varphi_k \colon K_1v_1 \hookrightarrow K_2v_2$ over $Kv$ sending $\underline a$ to $\underline a'$;
  likewise, there exists an elementary embedding $\varphi_\Gamma \colon v_1K_1 \hookrightarrow v_2K_2$ over $vK$ sending $\underline b$ to $\underline b'$.

  By Lemma~\ref{lem:emb-gen-subfield}, there is an embedding $\varphi \colon K_1 \hookrightarrow K_2$ over $K$ which induces $\varphi_k$ and $\varphi_K$.
  By Theorem~\ref{thm:RMC}, $\varphi$ is an elementary embedding.
  In particular, $\tp(\underline a, \underline b/K) = \tp(\underline a', \underline b'/K)$, as was to be shown.
\end{proof}

\begin{remark}
  Theorem \ref{thm:SE} means that the residue field sort and value group sort are stably embedded and orthogonal.
  Note that \cite[Example 11.5]{AJCohen} claims that the residue field of a finitely ramified
  field is not stably embedded. However, what the example in \cite{AJCohen} in fact 
  shows is that there are 
  subsets of the residue field which are
  $\emptyset$-definable in $(K,v)$ but are not invariant under automorphisms of the pure field $Kv$. 
  Thus, the traces on $Kv$ of $\emptyset$-definable sets in $(K,v)$ are not necessarily $\emptyset$-definable in the $\Lring$-structure $Kv$.
  In \cite[Definition 2.1.9]{CH}, this is referred to as $Kv$ not being canonically embedded 
  (and hence not fully embedded).
\end{remark}

\begin{corollary}\label{cor:dagger-defble}
  Let $(K,v)$ be a finitely ramified henselian valued field of initial ramification $e$.
  For all $d \geq 1$, $m \geq 0$,
  the set $\Omega_{d,e,m}(K,v)$ of Definition~\ref{def:Omega} is $\Lring$-definable in $Kv$ with parameters.
  In particular, the $\Ldagger{e}$-structure on $Kv$ is $\Lring$-definable with parameters.
\end{corollary}
\begin{proof}
  Since the sets $\Omega_{d,e,m}(K,v)$ are $\Lval$-definable in $(K,v)$, this is immediate from Theorem~\ref{thm:SE}.
\end{proof}

We now discuss the structure induced on the residue field and value group
across all finitely ramified henselian valued fields of given initial ramification index.
\begin{proposition}\label{prop:induced-str}
  Let $e \geq 1$, and let $T$ be the $\Lval$-theory of finitely ramified henselian valued fields of initial ramification index $e$.
  Let $\underline x$ be a tuple of variables of the residue field sort
  and $\underline y$ a tuple of variables of the value group sort.
  Every $\Lval$-formula $\varphi(\underline x, \underline y)$ is equivalent modulo $T$ to
  a boolean combination of formulas of the form $\widetilde\varphi_1(\underline x)$ and $\widetilde\varphi_2(\underline y)$,
  where $\varphi_1$ is an $\Ldagger{e}$-formula and $\varphi_2$ an $\Loag$-formula.
\end{proposition}
\begin{proof}
  By \cite[Corollary B.9.3]{TransseriesBook} it suffices to show that two complete $T$-types
  $p(\underline x, \underline y)$, $q(\underline x, \underline y)$ are equal if their restrictions to
  $\Ldagger{e}$-types on the residue field are equal and their restriction to $\Loag$-types on the value group are equal.
  This is Proposition \ref{prop:AKE-equiv-plus}.
\end{proof}

\begin{remark}\label{rem:induced-str}
  Proposition \ref{prop:induced-str} in particular means that for every $(K,v) \models T$,
  every $\Lval$-$\emptyset$-definable subset of the residue field $Kv$ is $\Ldagger{e}$-$\emptyset$-definable in $Kv$;
  in other words, the $\Ldagger{e}$-structure $Kv$ is canonically embedded in the sense of \cite[Definition 2.1.9]{CH}.
  Analogously, the $\Loag$-structure $vK$ is canonically embedded.
\end{remark}

\section*{Acknowledgements}
The authors would like to thank the Mathematical Sciences Research Institute (now SLMath) and the organizers of the MSRI programme ``Decidability, Definability
and Computability in Number Theory'' (DDC) in which we all participated. 
We began working on the material presented here
in 2020, during the virtual part of DDC, and continued at the in-person DDC reunion
in Berkeley in 2022.
Through the MSRI programme, this work was supported by the US National Science Foundation under Grant No.\ DMS-1928930.

Franziska Jahnke was
funded by the Deutsche Forschungsgemeinschaft (DFG, German Research Foundation) under Germany's Excellence Strategy EXC 2044-390685587, Mathematics M\"unster: Dynamics-Geometry-Structure,
as well as by a Fellowship from the Daimler and Benz Foundation.
Sylvy Anscombe and Franziska Jahnke were supported by GeoMod AAPG2019 (ANR-DFG).
Philip Dittmann was funded by the Deutsche Forschungsgemeinschaft (DFG) -- 404427454.

The authors extend their special thanks to both referees for their excellent suggestions which have substantially improved the paper.

\def\bibfont{\footnotesize}
\bibliographystyle{plain}

\end{document}